\documentclass[12pt,oneside,reqno]{amsart}

\usepackage{amsmath, amsthm, amsfonts, amsfonts,amssymb,amscd,amsmath,latexsym,enumerate,verbatim, calc, quotmark, mathtools, txfonts, calrsfs, imakeidx, fancyhdr, listings, lipsum, tkz-euclide, amsthm,datetime}
\usepackage{tikz-cd}
\usepackage{pifont}

\usepackage{mathtools}

\usepackage[T1]{fontenc}
\usepackage[utf8]{inputenc}

\usepackage[all, cmtip]{xy}

\usepackage[pagebackref]{hyperref}
\hypersetup{
     colorlinks = true, linkcolor = red,
     citecolor   = blue,
     urlcolor    = blue,
}

\newtheorem{innercustomgeneric}{\customgenericname}
\providecommand{\customgenericname}{}
\newcommand{\newcustomtheorem}[2]{%
  \newenvironment{#1}[1]
  {%
   \renewcommand\customgenericname{#2}%
   \renewcommand\theinnercustomgeneric{##1}%
   \innercustomgeneric
  }
  {\endinnercustomgeneric}
}

\newcustomtheorem{customthm}{Theorem}
\newcustomtheorem{customqn}{Question}

\def\NZQ{\mathbb}               

\def\QQ{{\NZQ Q}}

\def\Z{{\NZQ Z}}

\def\m{\mathfrak{m}}

\def\p{\mathfrak{p}}
\def\q{\mathfrak{q}}

{\tiny }

\DeclareMathOperator*{\Hom}{Hom}
\DeclareMathOperator*{\Ho}{H}

\DeclareMathOperator*{\Tor}{Tor}

\DeclareMathOperator*{\Zi}{Z}
\DeclareMathOperator*{\Bo}{B}

\DeclareMathOperator*{\coker}{coker}
\DeclareMathOperator*{\ann}{ann}
\DeclareMathOperator*{\soc}{soc}

\DeclareMathOperator*{\lo}{ll}
\DeclareMathOperator*{\edim}{edim}
\DeclareMathOperator*{\mx}{max}
\DeclareMathOperator*{\I}{I}
\DeclareMathOperator*{\init}{in}

\DeclareMathOperator*{\E}{E}

\newtheorem{theorem}{Theorem}[section]
\newtheorem{lemma}[theorem]{Lemma}

\newtheorem{proposition}[theorem]{Proposition}

\newtheorem{example}[theorem]{Example}

\newtheorem{definition}[theorem]{Definition}
\newtheorem*{definition*}{Definition}

\newtheorem*{acknowledgement}{Acknowledgement}

\newtheorem{theoremx}{\bf Theorem}

\title{A Study of Good and Bad Artinian Gorenstein local Rings}
\date{\today, \currenttime}

\author{Anjan Gupta, Shrikant Shekhar}

\address{Department of Mathematics\\
	Indian Institute of Science Education and Research Bhopal\\
	Bhopal Bypass Road, Bhopal, Madhya Pradesh, India. Pin - 462066.}
\email{anjan@iiserb.ac.in, shrikantshekhar21@gmail.com, }
\email{}

\thanks{Corresponding author: Anjan Gupta; 
	{\it email: anjan@iiserb.ac.in}}

{\tiny }

\begin{document}
\begin{abstract}
We say that a local ring $R$ is good, in the sense of Roos, if all finitely generated $R$-modules have rational Poincar\'e series that share a common denominator; otherwise, $R$ is said to be bad. An important class of good rings is the class of generalized Golod rings. In this paper, we show that connected sums of Artinian Gorenstein generalized Golod rings are good. We provide a criterion for decomposing Artinian Gorenstein local rings as connected sums. As a key application, we prove that a Gorenstein local ring $R$ with maximal ideal $\m$ is good under either of the following conditions: 
\begin{enumerate}
\item the multiplicity of $R$ is at most $12$ and its $h$-vector is different from $(1, 5, 5, 1)$,
\item
$\m^4$ = 0 and $\m^2$ is generated by at most four elements. 
\end{enumerate}
The above result records partial progress towards resolving a question posed by L.~Avramov.
We also present examples of bad Artinian Gorenstein local rings of any multiplicity greater than or equal to $18$. In all these cases, the results establishing that the rings are good are obtained by showing that the rings are generalized Golod rings.
\end{abstract}

\maketitle 



\section{introduction}
Let $R$ be a commutative Noetherian local ring with maximal ideal $\m$ and residue field $k = R/\m$. 
The Poincaré series of an $R$-module $M$ is defined as
\[
P_M^R(t) = \sum_{i = 0}^{\infty} \beta_i^R(M)\, t^i \in \mathbb{Z}[[t]],
\]
where $\beta_i^R(M) = \dim_k \Tor_i^R(M, k)$ denotes the $i$-th Betti number of $M$ over $R$. A central question in the homological study of local rings concerns the structure of the Poincaré series $P_k^R(t)$ of the residue field $k$. 
In the 1950s, Serre and Kaplansky independently asked whether the power series $P_k^R(t)$ is always a rational function; equivalently,  whether there exists a polynomial $g(t) \in \Z[t]$ such that $g(t) P_k^R(t) \in \Z[t]$. 
A strikingly similar question in algebraic topology, also attributed to Serre, concerns whether the Poincaré series associated with the loop space of a finite, simply connected CW complex is rational \cite{MR201468}. The resemblance between these two questions revealed a deep and surprising connection between algebra and topology \cite{roos1979relations}. It was widely conjectured that both questions might have affirmative answers, at least under suitable hypotheses.

For several years, no counterexamples were known, and the rationality of  Poincaré series $P^R_M(t)$ was proved in  certain special cases, including when $R$ is a regular local ring or a complete intersection. These positive results strengthened the belief that rationality might be a general phenomenon. However, in a breakthrough result, Anick~\cite{anick1982counterexample} constructed a local ring $R$ with residue field $k$ such that the Poincaré series  $P_k^R(t)$ is not rational.
Later, B{\o}gvad~\cite{bogvad1983gorenstein} refined Anick’s construction and produced an example of an Artinian Gorenstein local ring $R$ whose residue field $k$ has a non-rational Poincaré series $P_k^R(t)$. This was particularly striking because Artinian Gorenstein rings are known for their well-behaved properties. The existence of non-rational Poincaré series even in this class revealed that the rationality question is sensitive to more than just the overall algebraic structure of the ring—it depends on deeper, more intricate homological features.

These counterexamples not only resolved the original question in the negative but also opened a rich line of inquiry into the nature of Poincaré series for various classes of local rings. The rationality of the Poincar\'e series $P_k^R(t)$ implies that the Betti sequence $\{\beta_i^R(k)\}$ eventually satisfies a linear recurrence relation, offering significant information about the structure of the minimal free resolution of $k$ over $R$. A detailed account of applications arising from the rationality of Poincaré series is provided in \cite{avramov1994local}. Consequently, understanding when the Poincar\'e series is rational, and which structural properties of the ring influence this remains an important and active area of research in commutative algebra.

According to Roos, a local ring $R$ is said to be {good} if there exists a polynomial $d_R(t) \in \mathbb{Z}[t]$ such that, for every finitely generated $R$-module $M$, the product $d_R(t) P^R_M(t)$ lies in the polynomial ring $\mathbb{Z}[t]$~\cite[Definition~2.1]{MR2158761}. This definition implies that the Poincar\'e series of all finitely generated $R$-modules share a common polynomial denominator.
Roos constructed examples of standard graded Koszul algebras which fail to be good, thereby illustrating that even in seemingly structured algebraic settings, the behavior of Poincaré series can be unexpected ~\cite[Theorem~2.4]{MR2158761}. These examples are often referred to as {bad} rings, in contrast with good ones.

Despite the existence of such counterexamples, many well-known classes of local rings such as Golod rings \cite[Theorem~1]{ghione1975some} and local complete intersections \cite[Corollary~4.2]{gulliksen1974change} fall within the class of good rings.
Furthermore, certain homological bounds guarantee goodness: any local ring $R$ with codepth at most three is known to be good, and the same holds for Gorenstein local rings when the codepth is at most four~\cite[Theorem~6.4]{avramov1988poincare}; see also~\cite{avramov1998infinite} for further related results and developments. These findings highlight that while being a good ring is not universal, the property nevertheless holds
across a wide range of meaningful and well-studied classes of local rings.

The connected sum $R \# S$ of two Gorenstein local rings $R$ and $S$, sharing a common residue field $k$, was introduced and studied in detail by Ananthnarayan et al.
in~\cite{ananthnarayan2012connected}. The connected sum $R \# S$ is also a Gorenstein local ring with residue field $k$, and its behavior with respect to homological invariants has proven to be remarkably rich. They provided an explicit formula for the Poincar\'e series of the residue field $k$ over the connected sum $R \# S$, expressed in terms of the corresponding series over the original rings $R$ and $S$. The formula shows that if both $P^R_k(t)$ and $ P^S_k(t)$ are rational functions, then so is $P^{R \# S}_k(t)$. In other words, the rationality of the Poincaré series of the residue field is preserved under the connected sum operation, provided it holds for the components.

This naturally leads to a deeper and more subtle question: \emph{When are connected sums of good Gorenstein local rings also good?} That is, if $R$ and $S$ are good Gorenstein local rings, is the connected   sum $R \# S$ necessarily good as well? Avramov~\cite{avramov1994local} introduced a new class of local rings called generalized Golod rings. He proved that such rings are good. To date, no example of a good ring that is not generalized Golod has been found.

Motivated by our question and Avramov's result, we studied connected sums and their relation to generalized Golod rings. We prove  the following result.

\begin{theoremx}{\rm(Theorem \ref{ggc5})}\label{intro4}
	Let $R$ and $S$ be Artinian Gorenstein local rings with a common residue field $k$, and let $T = R \# S$ denote their connected sum. Fix an integer $l \geq 2$. Then $T$ is a generalized Golod ring of level $l$ if and only if both $R$ and $S$ are generalized Golod rings of level $l$.
\end{theoremx}

Another natural question is: When does an Artinian Gorenstein local ring admit a decomposition as a connected sum? This question was previously studied in \cite{ananthnarayan2019decomposing}. Recall that the Loewy length $\lo(R)$ of an Artinian local ring $R$ with maximal ideal $\m$ is defined as the largest integer $n$ for which $\m^n \neq 0$.  
 
We establish the following result. 

\begin{theoremx}{\rm(Theorem \ref{ac3})}\label{conde}
	Let $R$ be an Artinian Gorenstein local ring with maximal ideal $\m$ and with $\lo(R) \geq 3$. Then the following are equivalent:
	
	\begin{enumerate}
		\item
		There are Artinian Gorenstein local rings $S$, $T$ with $\lo(S) = \lo(R)$, $\lo(T) = 2$ such that $R = S \# T$.
		
		\item
		$(0 : \m^2) \not \subseteq \m^2$.
	\end{enumerate}
\end{theoremx}

Our work uncovers new families of Gorenstein local rings that are good, as detailed in the following theorem.

\begin{theoremx}{\rm(Theorem \ref{ggc6})}\label{intro5}
	Let $R$ be a Gorenstein local ring. Then $R$  is a generalized Golod ring, and in particular, good in the sense of Roos, in the following cases:
	\begin{enumerate}
		\item
		The square $\m^2$ of the maximal ideal $\m$ of $R$ is minimally generated by at most four elements and $\m^4 = 0$.
		\item
         The multiplicity of $R$ is at most $12$ and its $h$-vector is different from $(1, 5, 5, 1)$.
	\end{enumerate}
\end{theoremx}

Avramov \cite[Problem 1]{MR1165314} asked for conditions on a ring $R$ with multiplicity at most $12$ under which the Poincaré series $P_M^R(t)$ is rational for every $R$-module $M$. The preceding theorem offers partial progress toward this problem. It is worth noting that the same result implies that any Gorenstein local ring of multiplicity at most $11$ is good. 

A key step in proving Theorem \ref{intro5} is to decompose the ring in question as a nontrivial connected sum of Artinian Gorenstein local rings, invoking Theorem \ref{conde}. To illustrate the limitation of this approach, we exhibit a Gorenstein local ring of multiplicity $12$ with $h$-vector $(1,5,5,1)$ that cannot be expressed as a nontrivial connected sum of Artinian Gorenstein local rings (see Example \ref{noce}). This example demonstrates that our method cannot be extended to establish Theorem \ref{intro5}(2) in greater generality. Notably, Artinian Gorenstein local rings of small multiplicity need not be generalized Golod. In this regard, we prove the following result.

\begin{theoremx}{\rm(Theorem \ref{nonGGRAG})}\label{intro6}
	For every integer $n \geq 8$, there exists an Artinian Gorenstein local ring $S_n$ with $h$-vector $(1, n, n, 1)$ that is not generalized Golod. In particular, for every $m \geq 18$, there exists an Artinian Gorenstein local ring of multiplicity $m$ that fails to be generalized Golod.
\end{theoremx}

\noindent
{\bf Outline of the Paper}:
In Section~\ref{sec:pre}, we recall background material needed for the subsequent sections. Readers familiar with \cite{avramov1998infinite} and \cite{gulliksen1969homology} may skip this part. The main technical developments are presented in Sections \ref{sec:ben1} and \ref{sec:ben2}, where we examine a simple but central idea in depth. Consider a local ring $R$ and its quotient ring $S = R/I$ by an ideal $I$. In general, there is no established method to relate the acyclic closures of $R$ and $S$. 
We address this gap in a specific case, namely when $R$ is an Artinian Gorenstein local ring with embedding dimension at least two, and $I$ is  the socle of $R$. We prove that the acyclic closure of $S$ is a semifree DG algebra extension of the tensor product $ S \otimes_R R\langle X \rangle$, where $R\langle X \rangle$ denotes the acyclic closure of $R$ (see Proposition~\ref{ggc3}). An analogous construction for fiber products is developed in Proposition~\ref{ggfl3}, following a similar line of reasoning. The arguments presented in Section~\ref{sec:ben1} are based on a homological characterization of Golod algebras (see Theorem~\ref{prl7}) and make use of chain $\Gamma$-derivations defined on acyclic closures, a technique that can be traced back to Gulliksen’s foundational work (see Lemma~\ref{ggc2}).

Section~\ref{sec:ben2} contains the proof of Theorem \ref{intro4}. In Section~\ref{sec:ben3}, we study applications and implications of the preceding results, proving Theorems \ref{conde} and \ref{intro5}. We also establish Theorem \ref{intro6}, where we construct explicit examples of Artinian Gorenstein local rings, obtained via trivial extensions of the injective hull of the residue field, that fail to be good.

All rings in this article are Noetherian local rings with $1\neq0$. All modules are nonzero and finitely generated.
Throughout this article, the expression ``local ring $(R,\m, k)$" refers to a commutative Noetherian local ring $R$ with maximal ideal $\m$ and residue field $k = R/\m$. When information on the residue field is not necessary, we denote a local ring $R$ with maximal ideal $\m$ simply by $(R, \m)$.

\section{preliminaries}\label{sec:pre}
We begin with the necessary notation and background, laying the foundation for subsequent sections. For all other unexplained notations and terminology, we refer the reader to  \cite{avramov1998infinite}, \cite{bruns1998cohen}  and \cite{gulliksen1969homology}.

\begin{subsection}{DG algebras, DG$\Gamma$ algebras}\label{prl1}
Let $(R, \m, k)$ be a local ring. A DG algebra $(A, \partial)$ over the ring $R$ consists of a non-negatively graded, strictly skew-commutative $R$-algebra $A = \oplus_{i \geq 0} A_i$ such that $ A_0 = R/I$ for some ideal $I$ of $R$, together with an $R$-linear differential map $\partial : A \rightarrow A$ of degree $-1$ satisfying $\partial^2=0$ and the Leibniz rule (see \cite[Chapter 1, \S 1]{gulliksen1969homology}). A DG algebra $(A, \partial)$ is called augmented if it is equipped with a surjective $R$-algebra homomorphism $\epsilon : A \twoheadrightarrow R/J$ for some ideal $J$ of $R$ such that $\left. \epsilon \right |_{A_{\geq 1}} =  0$ and $\epsilon \circ \partial = 0$. The map $\epsilon$ is called an  augmentation map. 
The sets of cycles and boundaries of $A$ are denoted by $\Zi(A)$ and $\Bo(A)$ respectively.
We set $\I(A) = \ker {\epsilon}$,  called the augmentation ideal and $\I \Zi(A) = \I A \cap \Zi(A)$. Let $\Ho(A) = \Zi(A)/ \Bo(A)$. If $\tilde{\epsilon} : \Ho(A) \rightarrow R/ J$ is the induced map, then we define $\I\Ho(A) = \ker \tilde{\epsilon}$.

A DG algebra is called a DG$\Gamma$ algebra if to every element $x$ of even positive degree there is associated a sequence of elements $x^{(k)}$, $k \geq 0$ called the divided powers of $x$ satisfying the usual axioms (see \cite[Chapter \S 7]{gulliksen1969homology}) and compatible with the differential, i.e., $\partial(x^{(n)}) = \partial(x)(x^{(n - 1)})$.
The DG algebra (respectively, DG$\Gamma$ algebra) $A$ is called minimal if $\partial(A) \subseteq \m A$.

Let $f :  A \rightarrow  B$ be a morphism of DG$\Gamma$ algebras, i.e., $f$ is a chain homomorphism of complexes such that $f(xy) = f(x)f(y)$, $x, y \in A$, $f(1_A) = 1_B$ and $f(x^{(n)}) = {f(x)}^{(n)}$ for any element $x \in A$ of even positive degree.
An $A$-linear chain $\Gamma$-derivation of degree $n$ on the DG$\Gamma$ algebra $B$ is a $R$-linear morphism of complexes $\eta : B \rightarrow B[n]$  such that:
\begin{enumerate}
	\item
	$\eta \circ f = 0$ ($A$-linearity), 
	\item 
	$\eta$ satisfies Leibniz rule, i.e.,  $\eta(xy) = \eta(x)y + (-1)^{n\deg(x)}x\eta(y)$ for $x, y \in B$, 
	\item
	$\eta(x^{(i)}) = \eta(x) x^{(i-1)}$ for $x \in B$, $\deg(x)$ being even, 
	\item
	$\eta$ commutes with the differential $\partial_B$ of $B$ in the graded sense, i.e., $\eta \circ \partial_B = (-1)^n \partial_B \circ \eta$.
\end{enumerate}
 A DG$\Gamma$ algebra extension $A \hookrightarrow B$ is called semi-free if $B$ is obtained from $A$ by freely adjoining a set $X$ of $\Gamma$-variables, i.e., $B = A \langle X \rangle$ \cite[Construction 6.1]{avramov1998infinite}.
Throughout this article, unless specified otherwise, a DG algebra (respectively, a DG$\Gamma$ algebra) $A$ over $R$ is assumed to be an augmented DG algebra (respectively, DG$\Gamma$ algebra) with augmentation map $A \twoheadrightarrow k$.
\end{subsection}

\subsection{Acyclic closures and Tate resolutions} \label{prl2}
 Let $(R, \m, k)$ be a local ring and  $(A, \partial)$ be an augmented DG$\Gamma$ algebra over $R$ with augmentation map $\epsilon : A \twoheadrightarrow k$.  Assume further that each $\Ho_i(A)$ is a finitely generated $R$-module. Tate \cite{tate1957homology} constructed an augmented semi-free extension $ i : A \hookrightarrow A^* = A \langle X \rangle$, $X = \{X_i\}$ by a recurrent process of adjoining sets of variables (called $\Gamma$-variables) to kill cycles such that the augmentation map $\epsilon_{A^*} : A^* \rightarrow k$ satisfies $\epsilon_{A^*} \circ i = \epsilon_A$ and $A^*$ is acyclic, i.e., $\I \Ho(A^*) = 0$ (see \cite[\S 6.3.1]{avramov1998infinite} and \cite[Theorem 1.2.3]{gulliksen1969homology}). In the literature, $A^*$ is called the acyclic closure of $A$.

The acyclic closure satisfies $\Bo(A^*) \subseteq \I(A) A^*$  \cite[Theorem 1.6.2]{gulliksen1969homology}. Following \cite[\S 1.3]{avramov1994local}, the DG algebras $A\langle X_i : \deg(X_i) \leq d \rangle, d \geq 1$ obtained in the intermediate steps of adjunction of variables to construct the acyclic closure $A^* = A\langle X \rangle$ are called partial acyclic closures of $A$.

In the special case when $A=R$, the acyclic closure of the local ring $R$ is called the Tate resolution of $k$ over $R$. Note that the Koszul algebra $K^R$ of $R$ on a minimal set of generators of $\m$ can be identified with the partial acyclic closure $R \langle X_i : \deg(X_i) \leq 1\rangle$.

Let $A^* = R \langle W \rangle$, where $W = \{W_i\}_{i \geq 1}$ with $\deg(W_i) \leq \deg(W_j)$ for $i < j$, be the acyclic closure of $R$. We recall the notion of (normal) $\Gamma$-monomials from \cite[Remark 6.2.1]{avramov1998infinite}. For two $\Gamma$-monomials $M$ and $M'$ in $A^*$, there exists an integer $n$ large enough such that $M = W_1^{(c_1)} \ldots W_n^{(c_n)}$ and $M' = W_1^{(c'_1)} \ldots W_n^{(c'_n)}$ where $c_j, c'_j \geq 0$. We define a total order on the set of $\Gamma$-monomials by declaring  $M \prec M'$ if the last nonzero entry of $\big(c_1' - c_1, \ldots, c_n' - c_n, \deg(M') - \deg(M)\big)$ is positive. If $\beta$ is an element of $A^*$, then we can express $\beta$ as $\beta = \sum f_M M$, $f_M \in A$ where the sum is taken over finitely many $\Gamma$-monomials $M$. 
We define the initial $\Gamma$-monomial of $\beta$ as $\init(\beta) = \max \{M : f_M \neq 0 \}$.
The coefficient of $\init(\beta)$ in $\beta$ is called the leading coefficient of $\beta$. The following lemma is proved in \cite[Lemma 6.3.3]{avramov1998infinite}, of which statement (3) is a minor variant and admits an identical proof.

\begin{lemma}\label{ggc2}
Let $(R, \m, k)$ be a local ring and $(A, \partial)$ be an augmented DG algebra over $R$ with augmentation map $\epsilon : A \twoheadrightarrow k$.  Assume further that each $\Ho_i(A)$ is a finitely generated $R$-module. Let $A^* = A \langle W\rangle$, $W = \{W_i\}_{i \geq 1}$ be the acyclic closure of $A$ and $F^n(A^*) = A\langle W_i : \deg(W_i) \leq n \rangle$ denote a partial acyclic closure of $A$.
Then the following hold:
	\begin{enumerate}
		\item
		There exist $A$-linear chain $\Gamma$-derivations $\nu_j$, $1 \leq j \leq n$ such that $\nu_j(W_j) = 1$ and $\nu_j(W_i) = 0$ for $i < j$.
		
		\item
		$\nu_j(F^k(A^*)) \subseteq F^k(A^*)$ for $j \leq k$.
		
		\item
		For a $\Gamma$-monomial $M = W_1^{(c_1)} \ldots W_n^{(c_n)}$, define $\nu_M = \nu_1^{c_1} \circ \ldots \circ \nu_n^{c_n}$. Then $\nu_M(M) = \pm 1$ and $\nu_M(M') = 0$ for $M' \prec M$.
		
	\end{enumerate}
\end{lemma}
\begin{subsection}{Golod Algebras}
Now we assume that $A$ is a minimal DG$\Gamma$ algebra over a local ring  $(R, \m, k)$ such that $\Ho_0(A) = k$, $\Zi_{\geq 1}(A) \subseteq \m A$ and each $A_i$ is a free $R$-module of finite rank. In this setup, we present two equivalent definitions of Golod algebras. The first, rooted in Golod’s original work, involves trivial Massey operations. The second arises from a key inequality between formal power series,  which we took from \cite{levin1976lectures}. For readers eager to explore further characterizations, we recommend \textit{loc. cit.}.
\begin{enumerate}

\item The DG$\Gamma$ algebra $A$ is called a Golod algebra if $A$ admits a trivial Massey operation, i.e., there is a graded $k$-basis $\mathfrak{b}_R = \{h_{\lambda}\}_{\lambda \in \Lambda}$ of $\I \Ho(A)$, a function $\mu : \bigsqcup_{i = 1}^{\infty} \mathfrak{b}_R^i \rightarrow  A$ such that $\mu(h_\lambda) \in \I \Zi(A)$ with  $cls(\mu(h_\lambda))= h_\lambda$ and setting $\bar{a} = (-1)^{i+1}a$ for $a \in A_i$, one has 
\[ \partial \mu(h_{\lambda_1}, \ldots ,h_{\lambda_p}) = \sum_{j = 1}^{p-1}\overline{\mu(h_{\lambda_1}, \ldots ,h_{\lambda_j})}\mu(h_{\lambda_{j + 1}}, \ldots ,h_{\lambda_p}).\]

\vspace{1em}
\item
The DG$\Gamma$ algebra $A$ satisfies the following term-wise inequality of formal power series:
\[
P^R_k(t) \prec \frac{\Ho_{A \otimes k}(t)}{1 - t \left( \Ho_{\Ho(A)}(t) - 1 \right)},
\]
where $\Ho_{-}(t)$ denotes the Hilbert series. 
The DG$\Gamma$ algebra $A$ is said to be a Golod algebra if equality is attained \cite[Chapter~1,~\S4]{levin1976lectures}.
\end{enumerate}
\end{subsection}

\begin{subsection}{Homotopy Lie Algebras}
We assume, as before, that $A$ is a minimal DG$\Gamma$-algebra over a local ring $(R, \m, k)$ such that $\Ho_0(A) = k$, and each $A_i$ is a free $R$-module of finite rank. We define $\Tor^A(k, k) = \Ho(k \otimes_A A^*)$, $A^*$ being the acyclic closure of $A$ and $k \otimes_A A^* = \frac{k \otimes_R A^*}{\langle 1 \otimes a : a \in \ker \epsilon \rangle}$. We have a Hopf algebra structure on $\Tor^A(k, k)$ and the graded $k$-vector space dual $\Tor^A(k, k)^{\vee}$ is the universal enveloping algebra of a uniquely defined graded Lie algebra over $k$ (see \cite[Theorems 1.1, 1.2]{avramov1983local}). This Lie algebra is called the homotopy Lie algebra of $A$ and is denoted by $\pi(A)$. 

Set $\I\Tor^A(k, k)= \ker\{\tilde{\epsilon} : \Tor^A(k, k) \twoheadrightarrow k\}$. Let $\I\Tor^A(k, k)^{(2)}$ be the ideal of $\Tor^A(k, k)$ generated by products $ab$ for $a, b \in \I\Tor^A(k, k)$ and divided powers $x^{(2)}$ of elements $x$ of even positive degrees in $\I\Tor^A(k, k)$ . Then $\pi(A) = \Hom_k(\frac{\I\Tor^A(k, k)}{\I\Tor^A(k, k)^{(2)}}, k)$ (see \cite[\S 3.1]{avramov2006golod}). It follows that $\pi(A)$ is the graded vector space dual of the space of $\Gamma$-indecomposable elements in $\I\Tor^A(k, k)$.
\end{subsection}

The following result follows from a blend of \cite[Theorem 1.3]{levin1976lectures}, \cite[Theorem 2.3]{avramov2006golod}.

\begin{theorem}\label{prl7}
Let $(R, \m, k)$ be a local ring and $A$ be a  minimal DG$\Gamma$ algebra over $R$ such that $\Ho_0(A) =  k$ and each $A_i$ is a free $R$-module of finite rank. Let $A^*$ be the acyclic closure of $A$. Consider $A$ augmented naturally. 
Then the following are equivalent.
\begin{enumerate}
\item
The DG algebra $A$ is a Golod algebra.

\item
We have $\Zi_{\geq 1}(A) \subseteq \m A$ and the inclusion $i : \m A \hookrightarrow \m A^*$ induces an injective map $\Ho(\m A) \hookrightarrow \Ho(\m A^*) = \Tor^R(\m, k)$.

\item
 The map $\Ho(A \otimes_R k) \rightarrow \Ho(A^* \otimes_R k)=\Tor^R(k, k)$ is injective and $A$ admits a trivial Massey operation.

\item
The map $\Ho(A \otimes_R k) \rightarrow \Tor^R(k, k)$ is injective and the homotopy Lie algebra $\pi(A)$ is a free Lie algebra.

\end{enumerate}
\end{theorem}

The condition that the map $\Ho(A \otimes_R k) \rightarrow \Tor^R(k, k)$ is injective, is automatically satisfied when $A$ is a partial acyclic closure of a local ring $R$. 

\begin{definition}{\rm(Generalized Golod rings)}\label{ggr}
Let $R\langle X \rangle$ denote the acyclic closure of a local ring $R$. Avramov \cite[\S 1.7]{avramov1994local} defined a ring $R$ to be a generalized Golod ring of level $n$ if the partial acyclic closure $R\langle X_i : \deg(X_i) \leq n \rangle$ of $R$ is a Golod algebra, equivalently $\pi^{>n}(R)$ is a free Lie algebra. 
\end{definition}

Let $M$ be a module over a local ring $R$. The length of $M$ is denoted by $\ell(M)$ and the minimal number of generators of $M$ is denoted by $\mu(M)$. Moreover if $R$ is an Artinian ring, the socle of $R$ is defined as $\soc(R) = (0 : _R \m)$ and the Loewy length of $R$ as  $\lo(R) = \mx \{ n : \m^n \not=0\}$.

\begin{definition}\label{prl10}
Let $(R, \m_R, k)$, $(S, \m_S, k)$ be local rings with a common residue field $k$. Let $\pi_R : R \rightarrow k$, $\pi_S : S \rightarrow k$ be natural quotient maps from $R$, $S$ onto $k$ respectively. The fiber product of $R$ and $S$ over $k$ is defined as the ring $R \times_k S = \{(r, s) \in R \times S : \pi_R(r) = \pi_S(s) \}$. The rings $R$, $S$ are modules over $R \times_k S$ by left actions defined by projection maps $(r, s) \mapsto r$, $(r, s) \mapsto s$ respectively. 
      
Now we assume further that both $R$, $S$ are Artinian Gorenstein local rings with one dimensional socles $\soc(R) = \langle \delta_R \rangle$, $\soc(S) = \langle \delta_S \rangle$. Then the connected sum of $R$ and $S$ is defined as 
\[R \#  S = \frac{R \times_k S}{ \langle (\delta_R, - \delta_S)\rangle}.\]
We say a Gorenstein local ring $Q$ is decomposable as a connected sum if there are Artinian Gorenstein local rings $R$, $S$ such that $Q = R \#  S$, $\ell(R) < \ell(Q)$ and $\ell(S) < \ell(Q)$. 
\end{definition} 
We recall the following facts without specific reference.
\begin{enumerate}
\item The fiber product $T = R \times_k S$ is a local ring with maximal ideal $\m_R \oplus \m_S$ and the connected sum $Q = R \# S$ is an Artinian Gorenstein local ring whenever both $R$ and $S$ are.

\item Let $\m_T, \m_Q$ be the unique maximal ideals of $T$, $Q$ respectively. Then $\m^i_T = \m^i_R \oplus \m^i_S$ for $i \geq 1$. The maps $\m_R \rightarrow \m_Q$, $m \mapsto (\bar{m}, 0)$ and $\m_S \rightarrow \m_Q$, $n \mapsto (0, \bar{n})$ are injective.
We have $\lo(Q) = \max\{\lo(R), \lo(S)\}$ and $Q/ \m_Q^n  \cong R/ \m_R^{\min\{n, \lo(R)\}} \times_k S/ \m^{\min\{n, \lo(S)\}}_S$ for $2 \leq n \leq \lo(Q)$.

\item
 $\ell(R \times_k S) = \ell(R) + \ell(S) - 1$ and $\ell(R \#  S) = \ell(R) + \ell(S) - 2$.
\end{enumerate}

\section{Acyclic closures of $R/\soc(R)$ and $S \times_k T$}\label{sec:ben1}
Let $S$ and $T$ be local rings with common residue field $k$, and let $R$ be an Artinian Gorenstein local ring.
This section is devoted to the construction of the acyclic closure of $S \times_k T$ from those of $S$ and $T$, as well as the acyclic closure of $R/\soc(R)$ from that of $R$. These acyclic closures will play a crucial role in the subsequent section, where we examine the generalized Golod property of such rings.

\begin{lemma}\label{ggfl2}
Let $(S, \p, k)$ and $(T, \q, k)$ be two local rings, and set $R = S \times_k T$.  
Let $(S\langle X \rangle, \partial_1)$ and $(T \langle Y \rangle, \partial_2)$ be semi-free minimal DG$\Gamma$-algebra extensions of $S$ and $T$, respectively. Then there exist semi-free DG$\Gamma$ $R$-algebra extensions $(R\langle X \rangle, \tilde{\partial}_1)$ and $(R\langle Y \rangle, \tilde{\partial}_2)$ such that
\[
\tilde{\partial}_{1}(R\langle X \rangle) \subseteq \p R\langle X \rangle, \quad 
\tilde{\partial}_{2}(R\langle Y \rangle) \subseteq \q R\langle Y \rangle, \quad
S \otimes_R R\langle X \rangle \cong S\langle X \rangle, \quad
T \otimes_R R\langle Y \rangle \cong T\langle Y \rangle.
\]

Assume further that $\Zi_{\geq 1}(S \langle X \rangle) \subseteq \p S \langle X \rangle$ and $\Zi_{\geq 1}(T \langle Y \rangle) \subseteq \q T \langle Y \rangle$.  
Let $R\langle X \sqcup Y \rangle = R\langle X \rangle \otimes_R R\langle Y \rangle$. Then $R\langle X \sqcup Y \rangle$ is a Golod algebra if and only if both $S\langle X \rangle$ and $T\langle Y \rangle$ are Golod algebras.
\end{lemma}

\begin{proof}
We begin by constructing the semi-free extension $(R\langle X \rangle, \tilde{\partial}_1)$.
For $x \in X$, let $ \partial_1(x) = \sum p_{\underline{x}} \, \underline{x}$ in $S\langle X \rangle$. To define the differential $\tilde{\partial}_{1}$ on $R\langle X \rangle$, we set $\tilde{\partial}_1(x) = \sum i_1(p_{\underline{x}}) \underline{x}$ where $i_1 : \p \hookrightarrow \p \oplus \q$ is the natural inclusion and then extend $ \tilde{\partial}_1$ to a degree $(-1)$ map on $R\langle X \rangle$ by $R$-linearity and the Leibniz rule. It is clear that $(R\langle X \rangle, \tilde{\partial}_1)$ is a DG$\Gamma$  $R$-algebra satisfying $\tilde{\partial}_{1}(R\langle X \rangle) \subseteq \p R\langle X  \rangle$ and that there is an isomorphism of DG$\Gamma$  algebras $S \otimes_R R\langle X \rangle = S\langle X \rangle$. The construction of $(R\langle Y \rangle, \tilde{\partial}_2)$ proceeds analogously. 

 We have the following exact sequence of $R$-modules.
\[0 \rightarrow R \xrightarrow{\alpha} S \times T \xrightarrow{\beta} k \rightarrow 0,\]
where $\alpha(s, t) = (s, t)$ and $\beta(s, t) = s - t$. Tensoring this sequence with
 $R \langle X \sqcup Y \rangle = R\langle {X} \rangle \otimes_R R\langle Y \rangle$, we obtain the following short exact sequence of complexes of $R$-modules.

\[0 \rightarrow R \langle X \sqcup Y \rangle \xrightarrow{\alpha \otimes id} (S \otimes_R R \langle X \sqcup Y \rangle) \oplus (T \otimes_R R \langle X \sqcup Y \rangle) \xrightarrow{\beta \otimes id} k \otimes_R R\langle X \sqcup Y \rangle \rightarrow 0.\]

Here $id$ denotes the identity map on $R \langle X \sqcup Y \rangle$. Note that  $S \otimes_R R\langle {X \sqcup Y} \rangle = S\langle {X} \rangle \otimes_S S\langle Y \rangle$. The differential $S \otimes_R \tilde{\partial}_2$ is zero on $S\langle Y \rangle$. Therefore, $\Ho(S \otimes_R R \langle X \sqcup Y \rangle) = \Ho (S \langle X \rangle) \otimes_k k\langle Y \rangle$. Similarly we have $\Ho(T \otimes_R R \langle X \sqcup Y \rangle) = k\langle X \rangle  \otimes_k \Ho (T \langle Y \rangle)$. The differential on $k \otimes_R R\langle X \sqcup Y \rangle$ is zero. Therefore, the long exact sequence of homology takes the form  
\[ k \langle X \sqcup Y \rangle[+1] \xrightarrow{\partial} \Ho(R \langle X \sqcup Y\rangle) \xrightarrow{\Ho(\alpha \otimes id)} (\Ho (S \langle X \rangle) \otimes_k k\langle Y \rangle) \oplus (k\langle X \rangle  \otimes_k \Ho (T \langle Y \rangle)) \xrightarrow{\Ho(\beta \otimes id)} k \langle X \sqcup Y \rangle.\]

Let $\ker (\Ho(\beta \otimes id)) = K$ and $\coker (\Ho(\beta \otimes id)) = C$. Then we have the equality of Hilbert series: $\Ho_{\Ho(R\langle X \sqcup Y\rangle)}(t) = \frac{1}{t}\Ho_C(t) + \Ho_K(t)$ where 
$\Ho_{\Ho(R\langle X \sqcup Y\rangle)}(t) = \sum_{i \geq 0} \dim_k \Ho_i (R\langle X \sqcup Y\rangle)t^i$, 
$\Ho_C(t) = \sum_{i \geq 0} \dim_k C_it^i$ and $\Ho_K(t) = \sum_{i \geq 0} \dim_k K_i t^i$.

The quotient maps $\pi_S : S \rightarrow k$ and $\pi_T : T \rightarrow k$ induce maps $\pi_{S\langle X \rangle} : S\langle X \rangle \rightarrow k\langle X \rangle$ and $\pi_{T\langle Y \rangle} : T\langle Y \rangle \rightarrow k\langle Y \rangle$, respectively. Choose cycles $z \in \Zi(S\langle X \rangle)$, $z' \in \Zi(T\langle Y \rangle)$ and elements $u \in k\langle X \rangle$ and $v \in k\langle Y \rangle$. Let $[z]$, $[z']$ denote homology classes of $z$, $z'$, respectively. Note that $\Ho(\beta \otimes id)([z]\otimes v, u \otimes [z']) = \pi_{S\langle X \rangle}(z)v - u \pi_{T\langle Y \rangle}(z')$. 

By the given hypothesis $\Zi_{\geq 1}(S \langle X \rangle) \subseteq \p S \langle X \rangle$ and $\Zi_{\geq 1}(T \langle Y \rangle) \subseteq \q T \langle Y \rangle$, so $\Ho(\beta \otimes id)$ is zero on $(\Ho_{\geq 1} (S \langle X \rangle) \otimes_k k\langle Y \rangle) \oplus (k\langle X \rangle  \otimes_k \Ho_{\geq 1} (T \langle Y \rangle))$. The restriction of $\Ho(\beta \otimes id)$ on 
\[({\Ho}_0 (S \langle X \rangle) \otimes_k k\langle Y \rangle) \oplus (k\langle X \rangle  \otimes_k {\Ho}_0 (T \langle Y \rangle)) = k\langle Y \rangle \oplus k \langle X \rangle\] has a one-dimensional kernel generated by $(1, 1)$.
Therefore, we have 
\[K =  ({\Ho}_{\geq 1} (S \langle X \rangle) \otimes_k k\langle Y \rangle) \oplus (k\langle X \rangle  \otimes_k {\Ho}_{\geq 1} (T \langle Y \rangle)) \oplus k(1, 1) \ \text{and} \ C = \frac{k\langle X \sqcup Y \rangle}{k \oplus k\langle X \rangle_{\geq 1} \oplus k\langle Y \rangle_{\geq 1}}.\] Note that $k\langle {X \sqcup Y} \rangle = k\langle {X} \rangle \otimes_k k\langle Y \rangle$, so ${\Ho}_{k\langle X \sqcup Y  \rangle}(t) = \Ho_{k\langle X \rangle}(t) \Ho_{k\langle X \rangle}(t)$. Moreover
\begin{align*}
{\Ho}_{\Ho(R\langle X \sqcup Y\rangle)}(t) &= \frac{1}{t}{\Ho}_C(t) + {\Ho}_K(t)\\
&=\frac{1}{t}[{\Ho}_{k\langle X \rangle}(t){\Ho}_{k\langle Y \rangle}(t) - ({\Ho}_{k\langle X \rangle}(t) -1) - ({\Ho}_{k\langle Y \rangle}(t) - 1) -1]\\
& + ({\Ho}_{\Ho(S\langle X \rangle)}(t) -1){\Ho}_{k\langle Y \rangle}(t) + {\Ho}_{k\langle X \rangle}(t)({\Ho}_{\Ho(T\langle Y \rangle)}(t) -1) +1\\
&= 1 + {\Ho}_{k\langle Y \rangle}(t)({\Ho}_{\Ho(S\langle X \rangle)}(t) -1) + {\Ho}_{k\langle X \rangle}(t)({\Ho}_{\Ho(T\langle Y \rangle)}(t) -1) + \frac{1}{t}({\Ho}_{k\langle X \rangle}(t) -1)({\Ho}_{k\langle Y \rangle}(t) - 1).
\end{align*}
The above equality gives 
\[1 - t({\Ho}_{\Ho(R\langle X \sqcup Y \rangle)} (t) - 1) = {\Ho}_{k\langle X \rangle}(t)[ 1 - t  ({\Ho}_{\Ho(T\langle Y \rangle)}(t) - 1)] + {\Ho}_{k\langle Y \rangle}(t)[ 1 - t  ({\Ho}_{\Ho(S\langle X \rangle)}(t) - 1)] - {\Ho}_{k\langle X \rangle}(t){\Ho}_{k\langle Y \rangle}(t).\]

We have a termwise inequality of power series:
\begin{equation}\label{ggfle1}
P^S_k(t) \prec \frac{{\Ho}_{k\langle X \rangle}(t)}{1 - t ({\Ho}_{\Ho(S\langle X \rangle)}(t) - 1)} \ \text{and} \ P^T_k(t) \prec \frac{{\Ho}_{k\langle Y \rangle}(t)}{1 - t ({\Ho}_{\Ho(T\langle Y \rangle)}(t) - 1)}.
\end{equation}
\begin{align*}
\text{We have}\quad &\frac{1 - t({\Ho}_{\Ho(R\langle X \sqcup Y\rangle)} (t) - 1)}{{\Ho}_{k\langle X \sqcup Y \rangle}(t)}\\
&= \frac{{\Ho}_{k\langle X \rangle}(t)[ 1 - t ( {\Ho}_{\Ho(T\langle Y \rangle)}(t) - 1)] + {\Ho}_{k\langle Y \rangle}(t)[ 1 - t ( {\Ho}_{\Ho(S\langle X \rangle)}(t) - 1)] - {\Ho}_{k\langle X \rangle}(t){\Ho}_{k\langle Y \rangle}(t)}{{\Ho}_{k\langle X \rangle}(t){\Ho}_{k\langle Y \rangle}(t)}\\
&=\frac{1 - t  ({\Ho}_{\Ho(T\langle Y \rangle)}(t) - 1)}{{\Ho}_{k\langle Y \rangle}(t)} + \frac{1 - t  ({\Ho}_{\Ho(S\langle X \rangle)}(t) - 1)}{{\Ho}_{k\langle X \rangle}(t)} - 1\\
&\prec  \frac{1}{P^T_k(t)}  + \frac{1}{P^S_k(t)} - 1 \  \ \text{from \eqref{ggfle1}}\\
&=  \frac{1}{P^R_k(t)}.
\end{align*}

The DG$\Gamma$ algebra $R\langle X \sqcup Y \rangle$ is a Golod algebra if and only if the above inequality is an equality. This is equivalent to requiring that both inequalities in \eqref{ggfle1} are equalities, i.e., both $S\langle X \rangle$ and $T\langle Y \rangle$ are Golod algebras. 
\end{proof}

In what follows, we construct the acyclic closure of $R = S \times_k T$ by adjoining variables to kill the cycles of the DG algebra $R \langle X \sqcup Y \rangle$, as defined in the previous lemma.

\begin{proposition}\label{ggfl3}
Let $(S, \p, k)$ and $(T, \q, k)$ be local rings. Let $(S\langle X \rangle, \partial_1)$ and $(T\langle Y \rangle,  \partial_2)$ be acyclic closures of $S$, $T$ respectively. Set $R = S \times_k T$. Then $R$ admits an acyclic closure of the form $(R\langle X, Y, Z \rangle, \partial)$.
\end{proposition}

\begin{proof}
Set $U^n = S \langle X_i : \deg(X_i) \leq n \rangle$ and $V^n = T \langle Y_i : \deg(Y_i) \leq n \rangle$, $n \geq 1$. Then $U^n$ and $V^n$, $n \geq 1$ define filtrations of the acyclic closures $S\langle X \rangle$ and $T\langle Y \rangle$ respectively. The augmentation maps of $S\langle X \rangle$, $T\langle Y \rangle$ restrict to augmentation maps of $U^n$, $V^n$ respectively.
From the construction of acyclic closures, it follows that $\I \Ho_i(U^n) = \I \Ho_i(V^n) = 0$ for $0 \leq i \leq n - 1$. Moreover the homology classes of the cycles $\partial_1(X_i)$ with $\deg(X_i) = n + 1$ minimally generate $\I \Ho_n(U^n)$ and the homology classes of the cycles $\partial_2(Y_i)$ with $\deg(Y_i) = n + 1$ minimally generate $\I \Ho_n(V^n)$. 
We construct 

\[\tilde{U}^n = R \langle X_i : \deg(X_i) \leq n\rangle \ \text{and} \ \tilde{V}^n = R \langle Y_i : \deg(Y_i) \leq n\rangle,\ n \geq 1\]
as in Lemma \ref{ggfl2}. Let 
\[p^n : \tilde{U}^n \twoheadrightarrow \tilde{U}^n \otimes_R S = U^n \ \text{and} \ q^n : \tilde{V}^n \twoheadrightarrow \tilde{V}^n \otimes_R S = V^n, n \geq 1\] be natural quotient maps.

We construct inductively a sequence of augmented $DG\Gamma$ algebras $(F^nR, \partial^n)$ with natural augmentation $\epsilon^n : F^nR \twoheadrightarrow k$ satisfying the following properties: 
\begin{enumerate}
	\item
	$F^1R = K^R$, the Koszul complex of $R$,
	
	\item
	$\I\Ho_i(F^nR) = 0$ for $0 \leq i \leq  n - 1$,
	
	\item
	There are inclusions $\alpha^n : \tilde{U}^n \hookrightarrow F^nR$ and $\beta^n : \tilde{V}^n \hookrightarrow F^nR$ of complexes,
	
	\item
	There are surjective DG$\Gamma$ algebra homomorphisms $\phi^n : F^nR \twoheadrightarrow U^n$ and $\psi^n : F^nR \twoheadrightarrow V^n$ such that \[\phi^n \circ \alpha^n = p^n, \psi^n \circ \beta^n = q^n; (\phi^n \circ \beta^n)_{\geq 1} = 0, (\psi^n \circ \alpha^n)_{\geq 1} = 0 \ \text{and} \ (\phi^n \circ \beta^n)_{0} = pr_1, (\psi^n \circ \alpha^n)_{0} = pr_2,\]
	where $pr_1 : R \rightarrow S$ and $pr_2 : R \rightarrow T$ are natural projection maps.
\end{enumerate}

When $n = 1$, 
we define $F^1R = \tilde{U}^1 \otimes_R \tilde{V}^1 = R \langle X_i, Y_j : \deg(X_i) = \deg(Y_j) = 1 \rangle$, the Koszul complex on a minimal set of generators of $\p \oplus \q$ with the natural augmentation $\epsilon^1$. Define $\alpha^1 : \tilde{U}^1 \hookrightarrow F^1R$, $\beta^1 : \tilde{V}^1 \hookrightarrow F^1R$ to be natural inclusions and 
\[\phi^1 : F^1R \twoheadrightarrow \frac{F^1R}{(Y_i : \deg Y_i = 1)F^1 R} \otimes_R S = U^1, \quad \psi^1 : F^1R \twoheadrightarrow \frac{F^1R}{(X_i : \deg X_i = 1)F^1 R} \otimes_R T = V^1\] to be natural quotient maps. The properties $(1) - (4)$ are satisfied immediately.

Now we assume that $F^nR$, $n \geq 2$ is constructed with the aforesaid  properties. Let 
\begin{align*}
U^{n +1}  &= U^n \langle X_{c+1}, \ldots, X_{c + u} \rangle,  \partial_1(X_{c+i}) = s_i \ \text{for} \ 1 \leq i \leq u, \\
V^{n+1} &= V^n \langle Y_{d+1}, \ldots, Y_{d+v}\rangle, \partial_2(Y_{d+j}) = t_j \ \text{for} \ 1 \leq j \leq v.
\end{align*}

Then the homology classes of cycles $s_i, 1 \leq i \leq u$ form a basis of $\Ho_n(U^n)$ and homology classes of cycles $t_j, 1 \leq j \leq v$ form a basis of $\Ho_n(V^n)$. Note that cycles $s_i \in \p U^n$ and $t_j \in \q V^n$, so we can view cycles $s_i$, $t_j$ as elements in $\tilde{U}^n$ and $\tilde{V}^n$ respectively. We have 
\[\phi^n \circ \alpha^n (s_i) = s_i, \psi^n \circ \alpha^n (s_i) = 0, \psi^n \circ \beta^n(t_j) = t_j \  \text{and} \ \phi^n \circ \beta^n(t_j) = 0.\] It is now easy to see that homology classes of cycles $\alpha^n(s_i)$, $\beta^n(t_j)$,
$1 \leq i \leq u$, $1 \leq j \leq v$ are linearly independent in $\Ho_n(F^n R)$.

We choose cycles $z_k$, $1 \leq k \leq w$ in $\Zi_n(F^nR)$ such that homology classes of cycles $\alpha^n(s_i)$, $\beta^n(t_j)$, $z_k$, $1 \leq i \leq u$, $1 \leq j \leq v$, $1 \leq k \leq w$ form a basis of $\Ho_n(F^n R)$. We subtract a finite $R$-linear combination of $\alpha^n(s_i)$, $\beta^n(t_j)$ from $z_k$, if necessary,  to assume that $\phi^n(z_k)$  is a boundary in $U^n$ and $\psi^n(z_k)$ is a boundary in $V^n$. We choose $f_k \in U^n$, $g_k \in V^n$ of degree $n + 1$ such that $\partial_1(f_k) = \phi^n(z_k)$ and $\partial_2(g_k) = \psi^n(z_k)$. Since $p^n, q^n$ are surjective, we have $\tilde{f}_k \in \tilde{U}^n$,  $\tilde{g}_k \in \tilde{V}^n$ such that $p^n(\tilde{f}_k) = f_k$ and $q^n(\tilde{g}_k) = g_k$. We define 

\[F^{n +1} R = F^n R \langle X_{c + i}, Y_{d + j}, Z_k : 1 \leq i \leq u, 1 \leq j \leq v, 1 \leq k \leq w \rangle\] and extend the differential  $\partial^n$ from $F^n R$ to a differential $\partial^{n+1}$ on $F^{n+1} R$ by setting $\partial^{n+1}(X_{c + i}) = \alpha^n(s_i)$, $\partial^{n+1}(Y_{d + j}) = \beta^n(t_j)$ and $\partial^{n+1}(Z_j) = z_j$. 
The natural augmentation $\epsilon^n$ extends to the augmentation $\epsilon^{n+1} : F^{n+1}R \twoheadrightarrow k$. Clearly $\I \Ho_i(F^{n+1}R) = 0$ for $1 \leq i \leq n$.

By construction, the inclusions $\alpha^n$, $\beta^n$ extend to inclusions 
\[\alpha^{n + 1} : \tilde{U}^{n + 1} \hookrightarrow F^{n + 1}R \ \text{ and} \  \beta^{n + 1} : \tilde{V}^{n + 1} \hookrightarrow F^{n + 1}R\] of chain complexes respectively.
We only construct the map $\phi^{n+1}$. The construction of $\psi^{n+1}$ follows similarly. Let $I^n = \ker \phi^n$. Then 
\begin{align*}
	\phi^n \partial^{n+1}(Z_k - \alpha^n( \tilde{f}_k)) &= \phi^n \partial^{n+1}(Z_k) - \phi^n \partial^{n + 1} \alpha^n (\tilde{f}_k)\\
	&= \phi^n (z_k) - \partial_1 \phi^n \alpha^n  (\tilde{f}_k) = \phi^n (z_k) -  \partial_1 p^n  (\tilde{f}_k) = \phi^n (z_k) - \partial_1(f_k) = 0.
\end{align*} 
Therefore,  $\partial^{n+1}(Z_k - \alpha^n (\tilde{f}_k)) \in I^n$. We also have $\partial^{n+1}(Y_{d+j}) = \beta^n(t_j ) \subset I^n$.
It follows that 
\[I^{n + 1} = I^n + (Y_{d + j} ; 1 \leq j \leq v) + (Z_k - \alpha^n (\tilde{f}_k) ; 1 \leq k \leq w)\] is a DG$\Gamma$ ideal of $F^{n+1}R$.

Define $\phi^{n+1} : F^{n+1}R \rightarrow U^{n + 1}$ as the composition 
\[F^{n+1}R \rightarrow  \frac{F^{n+1}R}{I^{n +1}}  \rightarrow \frac{F^nR}{I^n}\langle X_{c+1}, \ldots, X_{c + u} \rangle \rightarrow U^{n +1}.\] Here the first map is the quotient map, the second map is induced by $Y_{d +j} \mapsto 0$, $Z_k \mapsto \alpha^n (\tilde{f}_k)$ and the last one is induced by the isomorphism $\frac{F^nR}{I^n} \cong U^n$.
Clearly,  $\phi^{n+1} \circ \alpha^{n+1} = p^{n+1}$ and $\phi^{n+1} \circ \beta^{n+1} = 0$. Thus, the induction step is complete.

By construction, $R\langle X, Y, Z \rangle = \varprojlim F^nR$ is the acyclic closure of $R$. This completes the proof. 
\end{proof}

The next lemma builds on the ideas of \cite[Theorem 2]{levin1978factoring} by Avramov and Levin.

\begin{lemma}\label{sbr1}
Let $(R, \m)$ be an Artinian Gorenstein local ring of embedding dimension $n \geq 2$. Let $R\langle X \rangle$ be the acyclic closure of $R$, $K^R = R \langle X_j : 1 \leq j \leq n\rangle$ be the Koszul complex of $R$ and  \[\mathfrak{X} = R \langle X_j : 1 \leq j \leq n + i \rangle, i \geq 1\] be a partial acyclic closure of $R$.  Let $\soc(R) = (t)$, $\bar{R} = \frac{R}{tR}$ and $\bar{ \mathfrak{X}} = \bar{R} \otimes_R \mathfrak{X}$. Choose  cycles $z \in \Zi_1(K^R)$ and $z' \in \Zi_{n-1}(K^R)$ such that $\partial{X_{n+1}} = z$ and $zz' = tX_1 \ldots X_n$. Define 
\[\mathfrak{Y} = \{ x \in \mathfrak{X} : z' \partial(x) \in \soc(R) X_1\ldots X_n\mathfrak{X}\}\] and  $\bar{\mathfrak{Y}} = \bar{R} \otimes_R \mathfrak{Y}$.
Then the following hold:
\begin{enumerate}
\item $\soc(R)X_1\ldots X_n \mathfrak{X} =  z' \Bo (\mathfrak{Y})$.
\item $\soc(R) \mathfrak{X} = (0 : \m^2)\Bo (\mathfrak{X}) + z' \Bo (\mathfrak{Y}) \subseteq \m \Bo (\mathfrak{X})$.
\item $\Zi(\bar{\mathfrak{X}}) = (0 : \m^2) \bar{\mathfrak{X}} + z' \bar{\mathfrak{Y}} + \bar{R} \otimes_R \Zi(\mathfrak{X})$.
\item The quotient map $\mathfrak{X} \twoheadrightarrow \bar{\mathfrak{X}} $ induces injective maps $\Ho(\mathfrak{X}) \rightarrow \Ho(\bar{\mathfrak{X}})$ and $\Ho(\m \mathfrak{X}) \rightarrow \Ho(\m \bar{\mathfrak{X}})$ on homology.
\item For any two cycles $v \in \I \Zi(\mathfrak{X})$, $w \in \Zi(z' \mathfrak{X})$, there exists an element $v' \in z' \mathfrak{X}$ such that $vw = \partial(v')$.
\item For any cycle $\bar{v} \in \I \Zi(\bar{\mathfrak{X}}), \bar{w} \in \Zi(z' \bar{\mathfrak{X}})$, there exists an element $\bar{v'} \in z' \bar{\mathfrak{X}} + (0 : \m^2)\bar{\mathfrak{X}}$ such that $\bar{v}\bar{w} = \partial(\bar{v'})$.
\end{enumerate}
\end{lemma}

\begin{proof}
Let $\partial(X_i) = x_i$, $1 \leq i \leq n$. Then the maximal ideal $\m$ is minimally generated by $\{x_1, \ldots, x_n\}$. 

It is proved in \cite[Lemma 1.2]{levin1978factoring} that 
\[\soc(R)K^R_i \subseteq (0 : \m^2) {\Bo}_{i} (K^R) \subseteq \m {\Bo}_{i} (K^R) \ \text{for} \ i = 1, \ldots, n - 1.\]
Furthermore, in line 10 on page 80 of the proof of \cite[Theorem 2]{levin1978factoring}, it is shown that
\[tX_1 \ldots X_n R\langle X \rangle = z' \partial(X_{n+1}) R\langle X \rangle \subseteq z' \Bo(R\langle X \rangle).\] 
By a very similar argument, we obtain
\[tX_1 \ldots X_n \mathfrak{X} = z' \partial(X_{n+1}) \mathfrak{X} \subseteq z' \Bo(\mathfrak{X}).\] 
It follows from definition of $\mathfrak{Y}$ that $tX_1 \ldots X_n \mathfrak{X} = z' \Bo(\mathfrak{Y})$. Therefore, proof of statement $(1)$ follows.

It is proved in equation 2.3 on page no 79 of the proof of \cite[Theorem 2]{levin1978factoring} that 
\[t R\langle X \rangle \subseteq (0 : \m^2) \Bo(R\langle X \rangle) + tX_1 \ldots X_n R \langle X \rangle.\] 
The same reasoning applies here and we obtain 
\[t \mathfrak{X} \subseteq (0 : \m^2) \Bo(\mathfrak{X}) + tX_1 \ldots X_n \mathfrak{X}.\] 
Using statement (1), this becomes $t \mathfrak{X} \subseteq (0 : \m^2) \Bo(\mathfrak{X}) +  z' \Bo(\mathfrak{Y})$. The reverse inclusion is clear, so statement (2) follows.

Both statements (3) and (4) follow from statement  (2). The argument is implicit in \cite[Theorem 2]{levin1978factoring}. We skip the details.

Now we prove statement (5). We may assume that $\deg(v) \geq 1$, since the result is obvious when $v \in \m$. Let $v = \sum_{i = 1}^nx_iv_i$, $v_i \in \mathfrak{X}$. Then we have 
\[v w = \sum_{i = 1}^n x_iv_iw  = \sum_{i = 1}^n[ \partial(X_iv_i)w  + X_i \partial(v_i) w ] 
 =\partial (v'_1)w  + \sum_{i = 1}^n X_i \partial(v_i) w \ \text{where} \ v'_1 = \sum_{i = 1}^n X_i v_i. 
\]
Since $z' \in K_{n -1}(R)$, we get  
\begin{equation}\label{soclexi}
X_i z'  \in X_1 \ldots X_n \mathfrak{X} \ \text{for} \ 1 \leq i \leq n. 
\end{equation}
By the hypothesis,  $w \in z' \mathfrak{X}$ and $X_1, \ldots, X_n$ are exterior variables, so
\begin{displaymath}\sum_{i = 1}^n X_i \partial(v_i) w  = X_1\ldots X_n u \ \text{for  some} \ u \in R \langle X_{n+1}, \ldots, X_{n+i} \rangle.\end{displaymath} 
It follows that $vw  = \partial(v_1')w  + X_1\ldots X_n u$. Since  $v, w $ are cycles, the element $X_1\ldots X_n u$ is also a cycle. Therefore, 
\[0 = \partial(X_1\ldots X_n u) = \sum_{i = 1}^n (-1)^{i - 1}x_i X_1 \cdots X_{i-1}X_{i+1} \cdots X_nu + (-1)^n X_1\ldots X_n\partial(u).\]

Note that $u \in R \langle X_{n+1}, \ldots, X_{n+i} \rangle$. Comparing coefficients of $X_1 \cdots X_{i-1}X_{i+1} \cdots X_n$ for $i = 1, \ldots, n$ on both sides, we obtain $u \in \soc(R) R \langle X_{n+1}, \ldots, X_{n+i} \rangle$. By statement $(1)$, we get
\[X_1 \ldots X_n u \in \soc(R) X_1\ldots X_n \mathfrak{X} \subseteq z' \Bo(\mathfrak{X}).\] Let $X_1 \ldots X_n u = \partial (v_2') z'$ for some $v_2' \in \mathfrak{X}$. Then we have 
\[vw = \partial(v_1')w + \sum_{i = 1}^n X_i \partial(v_i) w =  \partial(v_1')w  + X_1 \ldots X_n u = \partial(v_1')w + \partial(v_2')z' = \partial(v')\] where $v' = v'_1 w + v'_2 z'$. Clearly, $v' \in z' \mathfrak{X}$ and therefore statement (5) follows. 

Finally, we prove statement (6). We may assume that $\deg(\bar{v}) \geq 1$ to exclude the obvious case. Since $\bar{w} \in z' \bar{\mathfrak{X}} \subset \m \bar{\mathfrak{X}}$, $\m (0 : \m^2) \subseteq \soc(R)$ and $z'^2 = 0$,  we observe from statement (3) that $ \Zi(\bar{\mathfrak{X}}) \bar{w} = \bar{R} \otimes_R  \Zi(\mathfrak{X})\bar{w}$. Therefore, there exists $v_0 \in \Zi_{\geq 1}(\mathfrak{X})$ such that $\bar{v} \bar{w} = \bar{v}_0\bar{w}$.

Since $\bar{w} \in \Zi(z' \bar{\mathfrak{X}})$, we may choose a lift $w$ of $\bar{w}$ such that  $w = z'w' $ for some $w' \in \mathfrak{X}$. We observe that $t \mathfrak{X} \subseteq \Zi(\mathfrak{X})$. Therefore, by  statement (3), we obtain 
\begin{equation}\label{sbreq1}
z' w' = w = w_1 + z' w_2 + w_3 \ \text{for some} \  w_1 \in (0 : \m^2)\mathfrak{X}, w_2 \in \mathfrak{Y}\ \text{and} \ w_3 \in \Zi(\mathfrak{X}).
\end{equation}
It follows that $v_0w = v_0w_1 +  v_0z' w_2 +  v_0w_3.$

Since $v_0 \in \Zi_{\geq 1}(\mathfrak{X}) \subseteq \m \mathfrak{X}$, we deduce
\begin{equation}\label{bc1}
v_0 w_1 \in \soc(R)\mathfrak{X}.
\end{equation}
By statement (5), $v_0z' = \partial(v_1z')$ for some $v_1 \in \mathfrak{X}$. Therefore,
\[v_0 z' w_2 = \partial( v_1 z')w_2 = \partial( v_1 z' w_2) + (-1)^{n + \deg(v_1)} v_1 z' \partial(w_2).\]
Now, since $w_2 \in \mathfrak{Y}$, we have $z' \partial(w_2) \in \soc(R)X_1\ldots X_n \mathfrak{X}$ . This implies  
\begin{equation}\label{bc2}
v_0 z' w_2 = \partial( v_1 z' w_2) + w_4 \ \text{ where} \ w_4 = (-1)^{n + \deg(v_1)} v_1 z' \partial(w_2) \in \soc(R) \mathfrak{X}.
\end{equation}

Let $v_0 = \sum_{i=1}^n x_ia_i$ for $a_i \in \mathfrak{X}$, then we have
\[v_0 w_3 = \sum_{i=1}^n x_i a_iw_3 = \sum_{i=1}^n[\partial(X_ia_i)w_3 + X_i\partial(a_i)w_3] = \partial(v_2)w_3 + \sum_{i=1}^nX_i\partial(a_i)w_3 \ \text{where} \ v_2 = \sum_{i=1}^nX_ia_i.\]
 From equation \eqref{sbreq1} it follows that 
\[w_3 =  z'(w' - w_2) - w_1.\] 
By \eqref{soclexi}, we have $X_i z' \in X_1 \ldots X_n \mathfrak{X}$ and also $w_1 \in (0 : \m^2)\mathfrak{X}$. Therefore, 
 \[X_iw_3  = X_iz'(w' - w_2) - X_iw_1 \in  X_1 \ldots X_n \mathfrak{X} + (0 : \m^2)\mathfrak{X} \ \text{for} \ 1 \leq i \leq n.\]
Since $\partial(a_i) \in \m \mathfrak{X}$ and $X_1, \ldots, X_n$ are exterior variables, we obtain 
\[\sum_{i=1}^nX_i\partial(a_i)w_3 = X_1\ldots X_nu + w_5\] for some $u \in R\langle X_{n+1}, \ldots, X_{n+i} \rangle$ and $w_5 \in \soc(R)\mathfrak{X}$. 
It follows that
$v_0 w_3=  \partial(v_2)w_3 + X_1\ldots X_nu + w_5.$
Since $v_0, w_3, w_5$ are cycles, so is $X_1\ldots X_nu$. By an argument as seen before, we conclude that $u \in  \soc(R) R\langle X_{n+1}, \ldots, X_{n+i} \rangle$ and therefore  $X_1\ldots X_nu = \partial(v_3)z'$ for some $v_3 \in \mathfrak{Y}$ by statement (1). Consequently, we have  
 \begin{equation}\label{bc3}
v_0w_3 = \partial(v_2 w_3 + v_3z') + w_5.
\end{equation}

From \eqref{sbreq1}, \eqref{bc1}, \eqref{bc2}, and \eqref{bc3}, we obtain
\begin{align*}
v_0w &= v_0(w_1 + z' w_2 + w_3) = v_0w_1 + \partial( v_1 z' w_2) + w_4 + \partial(v_2 w_3 + v_3z') + w_5\\
&= (v_0w_1 + w_4 + w_5) + \partial(v_1z'w_2 + v_2 w_3 + v_3z').
\end{align*}
We have $v_0w_1 + w_4 + w_5 \in \soc(R)\mathfrak{X}$. Define 
$$v' = v_1z'w_2 + v_2 w_3 + v_3z' = v_1z'w_2 + v_2 [z'(w' - w_2) - w_1] + v_3z' 
= [v_1z'w_2 + v_2 z'(w' - w_2) + v_3z'] - v_2w_1.$$ 
Since $w_1 \in (0 : \m^2) \mathfrak{X}$, we have $v ' \in z' \mathfrak{X} + (0 : \m^2)\mathfrak{X}$.
Therefore, going modulo the socle, we deduce $\bar{v}\bar{w} = \overline{v_0w} = \partial(\bar{v'})$, where $v' \in z' \bar{\mathfrak{X}} + (0 : \m^2)\bar{\mathfrak{X}}$. Hence,  statement (6) follows.
\end{proof}

The following lemma serves as a key ingredient in the subsequent proposition, which constructs the acyclic closure of $R/\soc(R)$ for an Artinian Gorenstein local ring $R$.

\begin{lemma}\label{IL}
Let $(R, \m, k)$ be a local ring and $(A, \partial)$ be an augmented DG$\Gamma$ algebra over $R$ with augmentation map $\epsilon : A \twoheadrightarrow k$. Assume that each $\Ho_i(A)$ is a finitely generated $R$-module.
Let $B = A \langle X_1, \ldots, X_n \rangle$ be a partial acyclic closure and $J \subseteq \m A$ be a DG$\Gamma$ ideal of $A$ such that the following hold: 
\begin{enumerate}
\item 
$\partial(X_i) \in J B$ for all $i = 1, \ldots, n$,

\item
$J^2  = 0$,

\item 
For any $v \in \I\Zi(A)$ and $w \in \Zi(J)$, the product $v w$ is a boundary in $J$.
\end{enumerate}
Suppose $\deg X_1 \leq \deg X_2 \leq \ldots \leq \deg X_n = d$. 
Then the map $\Ho_i( A ) \rightarrow \Ho_i(B)$ is injective for all $i \geq d$ and the map $\Ho_i( \m A ) \rightarrow \Ho_i(\m B)$ is injective for all $i \geq 0$.
\end{lemma}

\begin{proof}
We first show that $\Ho_i(A) \to \Ho_i(B)$ is injective for all $i \geq d$. 
Fix $i \geq d$ and let $\alpha \in \Zi_i(A)$ be a boundary in $B$. Then $\alpha = \partial(\beta)$ for some $\beta \in B_{i+1}$.
Let $\init(\beta) = M$, where $M = X_{1}^{(c_1)} \cdots X_{r}^{(c_r)}$, with $r \leq n$, $c_j \geq 0$ for $j = 1, \ldots, r-1$ and $c_r \geq 1$.  By Lemma \ref{ggc2}, there exist $A$-linear chain $\Gamma$-derivations $\nu_j$, $1 \leq j \leq n$, such that $\nu_j(X_j) = 1$ and $\nu_j(X_i) = 0$ for $i < j$. The same lemma further gives $\nu_M(\beta) = \pm f$, where $f \in A$ is the leading coefficient of $\beta$, i.e., $f$ is the coefficient of $M$ in $\beta$. We  have $\nu_M(\alpha) = 0$, since $\alpha$ is free of $\Gamma$-variables $X_j$s. It follows that 
\[
\partial(f) = \pm \partial(\nu_M(\beta)) = \pm \nu_M(\partial(\beta)) = \pm \nu_M(\alpha) = 0.
\]
Therefore, $f \in \Zi(A)$.

Applying the same reasoning as in the case of $\nu_M(\beta)$, we conclude that $\nu_r(\beta)$ is a cycle in $B$. It follows that $\nu_r(\beta) \in \Zi(B) \cap B_{i+1- \deg X_r} \subseteq \Zi_{\geq 1}(B) \subseteq \Zi_{\geq 1}(A^*)$, where $A^*$ denotes the acyclic closure of $A$. We know  that $\Zi_{\geq 1}(A^*) \subseteq \Bo(A^*) \subseteq (\ker \epsilon) A^*$ \cite[Theorem 1.6.2]{gulliksen1969homology}, \cite[Theorem 6.3.4]{avramov1998infinite}. Therefore, $\nu_r(\beta) \in (\ker \epsilon) A^* \cap A^*_{\geq 1}$. 
The initial $\Gamma$-monomial of $\nu_r(\beta)$ is $\nu_r(M)$ and its leading coefficient is $\pm f$. Therefore, if $\deg(f) = 0$, we must have $f \in (\ker \epsilon)_0 = \m A_0$. Altogether, this discussion allows us to conclude that  $f \in \I \Zi(A)$.

It is given that $\partial(X_j) \in JB$ for $j = 1, \ldots, n$. Let $\partial(M) = \sum_{j = 1}^m g_jM_j$, where $g_j \in J$ for $j = 1, \ldots, m$ and $M_j$s are $\Gamma$-monomials in $X_1, \ldots X_r$ of total degree strictly less than $d$. We have $g_l \partial(X_j) = 0$ for $1 \leq j, l \leq m$ since $J^2 = 0$, so 
\[\partial^2(M) = \sum_{j = 1}^m \partial(g_j) M_j = 0.\] 
It follows that each $g_j \in \Zi(A) \cap J = \Zi(J)$, so by our hypothesis we obtain $f g_j =  \partial(h_j)$ for some $h_j \in J$, for $j = 1, \ldots, m$. Now
\[\partial(fM) = (-1)^{\deg{f}} \sum_{j = 1}^m f g_jM_j = (-1)^{\deg{f}} \sum_{j = 1}^m \partial(h_j) M_j = \partial(\delta)\] where $\delta = (-1)^{\deg{f}} \sum_{j = 1}^m h_jM_j \in JB \subseteq \m B$. Now we observe:
\[ \alpha = \partial(\beta) = \partial(\beta - f M) + \partial(fM) = \partial(\beta - f M) + \partial(\delta) = \partial(\beta - f M + \delta)\]
and moreover $\init(\beta - f M + \delta) < M =  \init(\beta)$. Repeating this argument iteratively, we obtain $\beta_0$ free of $X_1, \ldots, X_n$, equivalently $\beta_0 \in A$ such that $\alpha = \partial (\beta_0)$. Therefore, the homology class of $\alpha$ in $\Ho_i(A) $ is zero. Consequently, the injectivity of the first map is established. 

One can argue in the same way to prove that the second map is injective.  Here the  proof is slightly shorter as we do not need to carry out the extra work of showing that the leading coefficient of $\beta$, i.e.,  $f$ is in $\m A_0$ when $\deg(f) = 0$. 
\end{proof}

\begin{proposition}\label{ggc3}
Let $(R, \m, k)$ be an Artinian Gorenstein local ring with $\edim(R) = m \geq 2$ and $(R\langle X \rangle, \partial)$ be the acyclic closure of $R$. Set $\bar{R} = R/ \soc(R)$. Then $\bar{R}$ admits an acyclic closure of the form $\bar{R} \langle X, Y, Z \rangle$ with the following properties.
\begin{enumerate}
\item $Y = \{Y_i : i \geq 1\}$, $Z=\{Z_j : j \geq 1\}$, $\partial(Y_i) \in (0 : \m^2) \bar{R} \langle X, Y, Z \rangle$ and $\partial(Z_j) \in z' \bar{R} \langle X, Y, Z \rangle$ for some $z' \in \Zi_{m-1}(K^R) \setminus \Bo_{m-1}(K^R)$.
\item The  $R$-linear map $R\langle X \rangle \rightarrow \bar{R}\langle X, Y, Z \rangle$ sending $X_i$ to $X_i$ induces the following injective maps on homology: 
\begin{align*}
&\Ho(R\langle X_i : \deg(X_i) \leq n \rangle) \hookrightarrow \Ho(\bar{R}\langle X_i, Y_j, Z_k : \deg(X_i), \deg(Y_j), \deg(Z_k) \leq n \rangle) \ \text{and} \ \\
&\Ho(\m R\langle X_i : \deg(X_i) \leq n \rangle) \hookrightarrow \Ho(\m \bar{R}\langle X_i, Y_j, Z_k : \deg(X_i), \deg(Y_j), \deg(Z_k) \leq n \rangle) \ \text{for all $n \geq 2$.}
\end{align*}

\end{enumerate}
\end{proposition}

\begin{proof}
Let  
\[\soc(R) = (s),  \bar{R} = R/ (s), U^n = R \langle X_i : \deg(X_i) \leq n \rangle \  \text{and} \ \bar{U}^n = \bar{R} \otimes_R U^n \ \text{for} \ n \geq 1.\] 

The DG$\Gamma$ algebras $U^n, \bar{U}^n$ admit natural augmentation maps. By construction,  $\I \Ho_i(U^n) = 0$ for $0 \leq i  \leq n - 1$ and the homology classes of the cycles $\partial(X_i)$ with $\deg(X_i) = n +1$ minimally generate $\I \Ho_n(U^n)$. Now $\deg(X_i) = 1$ for $1 \leq i \leq m$ where $m = \dim {\m}/{\m^2} = \edim(R)$. Let $\partial(X_i) = x_i$ for $1 \leq i  \leq m$. Then $\m$ is minimally generated by $\{x_i :1 \leq i \leq m\}$. Clearly, $U^1$ is the Koszul complex $K^R$ on $x_1, \ldots, x_m$. Let $z \in \Zi_1(K^R) \setminus \Bo_1(K^R)$ be such that  $\partial(X_{m+1}) = z$. We choose $z' \in \Zi_{m-1}(K^R)$ such that $zz' = sX_1\cdots X_m$ \cite{avramov1971homology}.

We construct inductively a sequence $V^n$ of DG algebras over $\bar{R}$ with the following properties.
\begin{enumerate}
\item $V^1 = K^{\bar{R}}$, the Koszul complex of $\bar{R}$.

\item
$\I \Ho_i(V^n) = 0$ for $0 \leq i \leq n -1$.

\item   $V^n = \bar{U}^{n} \langle Y_{p_{n-1} + 1}, \ldots, Y_{p_{n}},  Z_{q_{n-1} + 1}, \ldots, Z_{q_{n}} \rangle$, 
$n \geq 2$ for some nondecreasing sequences $\{p_i\}$, $\{q_i\}$ of integers. The homology classes of $\partial(X_i)$ with $\deg(X_i) = n$, together with $\partial(Y_j)$ and $\partial(Z_k)$ for $p_{n-1}+1 \leq j \leq p_n$ and $q_{n-1}+1 \leq k \leq q_n$, minimally generate $\Ho_{n-1}(V^{n-1})$.
\item
$\partial(Y_j) \in (0 : \m^2) V^{n-1}$ and $\partial(Z_k) \in z' V^{n-1}$ for $p_{n-1} + 1 \leq j \leq p_n$ and $q_{n-1} + 1 \leq k \leq q_n$.

\item
The composition $U^n \twoheadrightarrow \bar{U}^n \hookrightarrow V^n$ induces injective maps 
\[\Ho(U^n) \hookrightarrow \Ho(V^n) \ \text{and} \  \Ho(\m U^n) \hookrightarrow \Ho(\m V^n) \ \text{for} \ n \geq 2. \]
\end{enumerate}

Note that $\bar{U}^1$ is the Koszul complex of $\bar{R}$.  We set 
\[V^1 = \bar{U}^1 \  \text{with its natural augmentation.}\]
 By \cite[Theorem 1]{levin1978factoring}, if $\bar{H} = \frac{\Ho(K^R)}{\Ho_m(K^R)}$, $\bar{K} = \frac{K^R \otimes_R k}{K_m^R \otimes_R k}[-1]$, then $\Ho(K^{\bar{R}}) = \bar{H} \lJoin \bar{K}$. It follows that $\Ho_1(V^1) = \Ho_1(U^1) \oplus k$. The pairing 
\[\frac{\m}{\m^2} \times \frac{(0 : \m^2)}{(0 : \m)} \rightarrow \soc(R), (\bar{x}, \bar{y}) \mapsto xy\] is non-degenerate. 
We choose $y_i \in (0 : \m^2)$, $1 \leq i \leq m$ such that 
\[x_iy_j = 
\begin{cases}
0 \ \text{for $i \not= j$}\\
s \ \text{for $i = j$.}
\end{cases}
\]
The homology classes of $\partial(X_i)$ with $\deg(X_i) = 2$, together with $\bar{y}_1X_1$, minimally generate $\Ho_1(V^1)$. Define  
\[V^2 = \bar{U}^2 \langle Y_1 : \partial(Y_1) = \bar{y}_1X_1 \rangle\] with its natural augmentation. 
Here $p_0 = p_1 = 0$, $p_2 = 1$, $q_0 = q_1 = q_2 = 0$. 

We observe that $\Ho_0(\bar{U}^2) = k$ and  $\displaystyle {\Ho}_1(\bar{U}^2) = \frac{\Ho_1(\bar{U}^1)}{ \langle cls(\partial(X_i)) : \deg(X_i) = 2 \rangle}$, which  is minimally generated by the homology  class of $\bar{y}_1X_1$.  It follows that $V^2$ is a partial acyclic closure of $\bar{U}^2$. Setting $A = \bar{U}^2$, $B = V^2$ and $J = \bar{y}_1X_1 A$, we see that  $J$ is a DG$\Gamma$ ideal of $A$ such that $J^2 = 0 $. Moreover, the product of any cycle in $\I \Zi(A)$ with a cycle in $\Zi(J)$ is zero since $\m J = 0$. Thus, all the hypotheses of Lemma \ref{IL} are satisfied. Therefore, it follows that  
\[{\Ho}_i(\bar{U}^2) \hookrightarrow {\Ho}_i(V^2) \ \text{and} \  {\Ho}_j(\m \bar{U}^2) \hookrightarrow {\Ho}_j(\m V^2)\]
are injective maps for all $i \geq 2$ and $j \geq 0$. On the other hand, $(4)$ of Lemma \ref{sbr1} shows that the maps 
\[{\Ho}_i(U^2) \hookrightarrow {\Ho}_i(\bar{U}^2) \ \text{and} \  {\Ho}_i(\m U^2) \hookrightarrow {\Ho}_i(\m \bar{U}^2)\] 
are injective for all $i \geq 0$. We observe that $\Ho_0(U^2) = k$ and $\Ho_1(U^2) = 0$. Therefore,
by taking compositions, we conclude that 
\[\Ho(U^2) \hookrightarrow \Ho(V^2) \ \text{and} \  \Ho(\m U^2) \hookrightarrow \Ho(\m V^2)\]
are injective maps. 

Now we assume that $V^n, n \geq 2$ is constructed with aforesaid properties. We proceed to construct $V^{n+1}$. Since $V^n$ is a partial acyclic closure of $\bar{R}$, we have $\I \Zi(V^n) \subseteq \m V^n$. \\

\noindent
{\bf Claim} : { ${\Zi}_n(V^n) = {\Bo}_n(V^n) +  \bar{R} \otimes_R \Zi_n(U^n)  + (0 : \m^2) V_n^n + z'V^n \cap {\Zi}_n(V^n), n \geq 2.$} 

Set \[S = {\Bo}_n(V^n) +  \bar{R} \otimes_R {\Zi}_n(U^n)  + (0 : \m^2) V_n^n + z'V^n \cap {\Zi}_n(V^n).\] Clearly, $S \subseteq \Zi_n(V^n)$.
To prove the reverse inclusion, we choose $\alpha \in \Zi_n(V^n)$. 
Then $\alpha$ is a finite linear combination of $\Gamma$-monomials in $Y_j$ and $Z_k$ with coefficients in $\bar{U}^n$. Let the total degree of $\alpha$ in $Y_i$ and $Z_k$ be $d$. We prove $\alpha \in  S$ by induction on $d$. If $d = 0$, then $\alpha$ has no terms involving $Y_j$ and $Z_k$, so $\alpha \in \Zi_n(\bar{U}^n)$. By (3) of Lemma \ref{sbr1}, we obtain  $\alpha \in \bar{R} \otimes_R \Zi_n(U^n) + (0 : \m^2) \bar{U}_n^n + z' \bar{U}^n \cap  \Zi_n(\bar{U}^n)$. Since $\bar{U}^n \subseteq V^n$, we have $\alpha \in S$.

Now we assume that $d > 0$. We choose a $\Gamma$-monomial term in $Y_j, Z_k$ of the form $f M$ for $f \in \bar{U}^n$, $M= Y_{l_1}^{(c_1)}\ldots Y_{l_s}^{(c_s)}Z_{m_1}^{(d_1)} \ldots Z_{m_t}^{(d_t)}$  in the expression of $\alpha$ such that $\deg(M) = d$, $\deg(f) = e$. Clearly $n = d + e$ and so $e < n$. 
Since $\alpha$ is a cycle and $M$ is a $\Gamma$-monomial of highest total degree in $Y_j, Z_k$, $f$ is a cycle in $\bar{U}^n$. Therefore,  by (3) of Lemma \ref{sbr1}, we have  
\[f \in \bar{R}\otimes_R {\Zi}_e(U^n) + (0 : \m^2)\bar{U}_e^n + {\Zi}_e(\bar{U}^n) \cap z' \bar{U}^n.\] 

Since $V^n$ is a partial acyclic closure of $\bar{R}$, we have $\alpha \in \Zi_n(V^n) \subseteq \m V^n$, $n \geq 2$. Therefore, if $e = 0$, we get $f \in \m \bar{R} = \Bo_0(\bar{U^n})$. In contrast, if $1 \leq e < n$, we get $\Zi_e(U^n) = \Bo_e(U^n)$.   Therefore, we may  write 
 \[f = \partial(g) + f' + f'' \ \text{for} \  g \in \bar{U}^n_{e+1}, f' \in  (0 : \m^2)\bar{U}_e^n \ \text{and} \ f'' \in {\Zi}_e(\bar{U}^n) \cap z' \bar{U}^n.\] 
Hence,  \[fM = \partial(gM) + f'M + f''M + (-1)^eg\partial(M).\] 

Note that $f'M \in (0 : \m^2)V_n^n$. By property (4) of $V^n$, we obtain $z' \partial(M) = 0$, which implies $f''M \in z'V^n \cap \Zi_n(V^n)$. We observe that  $g \partial(M)$ has total degree in $Y_j$ and $Z_k$ strictly less than $d$ and all other terms in the expression of $fM$ are in $S$. 
 
 If we write each $\Gamma$-monomial term of total degree $d$ in $Y_j$, $Z_k$ appearing in the expression of $\alpha$ in this manner, we have $\alpha = \alpha_1 + \alpha_2$, where $\alpha_1 \in S$, $\alpha_2 \in \Zi_n(V^n)$ and the total degree of $\alpha_2$ in $Y_j$, $Z_k$ is strictly less than $d$. By the induction hypothesis, $\alpha_2 \in S$, so $\alpha \in S$ and our claim follows.

The homology classes of cycles $\partial(X_i)$ with $\deg(X_i) = n+1$ minimally generate $\Ho_n(U^n)$. Because of our claim and the fact that $\Ho(U^n) \hookrightarrow \Ho(V^n)$ is injective,
we can adjoin $\Gamma$-variables $X_i$ for $\deg(X_i) = n+1$ and 
$Y_j$, $Z_k$ for $p_n + 1 \leq j \leq p_{n+1}$, $q_n + 1 \leq k \leq q_{n+1}$ to $V^n$ such that the homology classes of $\overline{\partial(X_i)}$, $\partial(Y_j)$, $\partial(Z_k)$ minimally generate $\Ho_n(V^n)$ and $\partial(Y_j) \in (0 : \m^2)V^n$, $\partial(Z_k) \in z'V^n$. We define 
\begin{align*}
V^{n+1} &= V^n \langle X_i, Y_j, Z_k : \deg(X_i) = n+1, p_n + 1 \leq j \leq p_{n+1}, q_n + 1 \leq k \leq q_{n+1} \rangle \\
&= \bar{U}^{n+1} \langle Y_j, Z_k : 1 \leq j \leq p_{n+1}, 1 \leq k \leq q_{n+1} \rangle.
\end{align*}
with the natural augmentation. Clearly, $\I \Ho_{i}(V^{n+1}) = 0$ for $0 \leq i \leq n$.

We rename $Y_j$, $Z_k$ as $W_l$, $1 \leq l \leq p_{n+1} + q_{n+1}$ such that $\deg(W_l) \leq \deg(W_{l'})$ for $l < l'$. This is possible because for any positive integer $d$, there are only finitely many $Y_j$, $Z_k$ of degree $d$. 
Now $V^{n+1} = \bar{U}^{n+1}\langle W_l : 1 \leq l \leq p_{n+1} + q_{n+1}\rangle$ is a DG algebra over $\bar{U}^{n+1}$. By  construction,
the homology classes of $\partial(W_l)$, $\deg(W_l) = d$ minimally generate 
\[\I {\Ho}_{d-1}(\bar{U}^{n+1}\langle W_l :  \deg(W_l) \leq d - 1 \rangle) = \I {\Ho}_{d-1}(V^{d-1})/ \langle cls(\partial(X_i)) : \deg(X_i) = d \rangle\]
 for $1 \leq d \leq n + 1$.  It follows that $V^{n+1}$ is a partial acyclic closure of $\bar{U}^{n+1}$. Let $A = \bar{U}^{n+1}$, $B = V^{n+1}$ and $J = (0 : \m^2) A + z' A$. Then $J$ is a DG$\Gamma$ ideal of $A$ such that $J^2 =  0$. We observe that 
 \[\Zi(J) = (0 : \m^2) A + \Zi(z' A).\]
 By (6) of Lemma \ref{sbr1}, the product of any cycle in $\I \Zi(A)$ with a cycle in $\Zi(J)$ is a boundary in $J$. Applying Lemma \ref{IL}, we conclude that 
 \[{\Ho}_i(\bar{U}^{n + 1}) \hookrightarrow {\Ho}_i(V^{n + 1}) \ \text{and} \  {\Ho}_j(\m \bar{U}^{n + 1}) \hookrightarrow {\Ho}_j(\m V^{n + 1})\]
are injective maps for all $i \geq n + 1$ and $j \geq 0$.  On the other hand, $(4)$ of Lemma \ref{sbr1} implies that the maps 
\[{\Ho}_i(U^{n + 1}) \hookrightarrow {\Ho}_i(\bar{U}^{n + 1}) \ \text{and} \  {\Ho}_i(\m U^{n + 1}) \hookrightarrow {\Ho}_i(\m \bar{U}^{n + 1})\] 
are injective for all $i \geq 0$. We observe that $\Ho_0(U^{n + 1}) = k$ and $\Ho_i(U^{n + 1}) = 0$ for $1 \leq i \leq n$. Thus, taking compositions, we conclude that 
\[\Ho(U^{n + 1}) \hookrightarrow \Ho(V^{n + 1}) \ \text{and} \  \Ho(\m U^{n + 1}) \hookrightarrow \Ho(\m V ^{n + 1})\]
are injective maps. 

Therefore, our induction step of construction of $V^{n+1}$ is complete. We define the acyclic closure of $\bar{R}$ as $\bar{R} \langle X, Y, Z \rangle = \projlim V^{n}$. The DG$\Gamma$ algebras $\{V^n\}$ define a filtration of $\bar{R} \langle X, Y, Z \rangle$. The properties (1) and (2) stated in the theorem are satisfied by $\bar{R} \langle X, Y, Z \rangle$ by  construction.
\end{proof}

\section{Main results}\label{sec:ben2}
We now have the necessary tools to prove our main results. 
We recall the definition of large homomorphisms  which will be useful in the sequel. 
A surjective homomorphism $f : R \twoheadrightarrow S$ of local rings is called large if the induced map  $f^* :  \pi(S) \rightarrow \pi(R)$ between homotopy Lie algebras is injective. The following result is crucial for proving Theorem~\ref{ggfl4}.

\begin{lemma}\label{largeGGL}
Let $\phi : R \rightarrow S$ be a large homomorphism of local rings. If $R$ is a generalized Golod ring, then $S$ is also a generalized Golod ring.
\end{lemma}

\begin{proof}
Since $\phi : R \rightarrow S$ is a large homomorphism, the induced map  $\phi^* : \pi(S) \rightarrow \pi(R)$ between homotopy Lie algebras is injective. It follows that  $\pi^{>n}(S)$ is a subalgebra of $\pi^{>n}(R)$ which is free because $R$ is a generalized Golod ring of level $n$. A graded sub-Lie algebra of a free graded Lie algebra on a positively graded vector space is free \cite[Proposition A.1.10]{lemaire2006algebres}. Therefore, $\pi^{>n}(S)$ is free and  consequently $S$ is a generalized Golod ring of level $n$.
\end{proof}

The following theorem  presents a generalization of Lescot's result \cite[Theorem 4.1]{lescot1983serie}.

\begin{theorem}\label{ggfl4}
Let $(S, \mathfrak{p}, k)$ and $(T, \mathfrak{q}, k)$ be local rings, and set 
$R = S \times_k T$. Then $R$ is a generalized Golod ring of level $n$ if and 
only if both $S$ and $T$ are generalized Golod rings of level $n$.
\end{theorem}

\begin{proof}
Since $p_1 : R \rightarrow S$ is a large homomorphism \cite[Theorem 3.4]{moore2009cohomology},  the "only if" part follows from Lemma~\ref{largeGGL}.

Now we assume that both $S$ and $T$ are generalized Golod rings of level $n$. 
Let $S\langle X \rangle$ and $T \langle Y \rangle$ be acyclic closures of $S$ and $T$ respectively. Then both $S\langle X_i : \deg(X_i) \leq n \rangle$ and $T\langle Y_i : \deg(Y_i) \leq n \rangle$ are Golod algebras, so by Lemma \ref{ggfl2}, $U = R\langle X_i, Y_j : \deg(X_i), \deg(Y_j) \leq n \rangle$ is a Golod algebra, i.e., $\pi(U)$ is a free Lie algebra. The acyclic closure of $R$ is of the form  $R\langle X, Y, Z  \rangle$ by Proposition \ref{ggfl3}. Set $V = R\langle X_i, Y_j, Z_k : \deg(X_i), \deg(Y_j), \deg(Z_k) \leq n \rangle$.

The inclusion of DG$\Gamma$ algebras $ i : U \hookrightarrow V$  induces the map $i_* : \Tor^U(k, k) \rightarrow \Tor^V(k, k)$. Since $R\langle X, Y, Z  \rangle$ is a semi‑free extension of both $U$ and $V$, which provides a minimal free resolution of $k$, it follows that
\[{\Tor}^U(k, k) = \Ho(k \otimes_U R\langle X, Y, Z  \rangle) = k \otimes_U R\langle X, Y, Z  \rangle\ \text{and} \ {\Tor}^V(k, k) = \Ho(k \otimes_V R\langle X, Y, Z  \rangle) = k \otimes_V R\langle X, Y, Z  \rangle.\]
 Therefore, $i_*$ is surjective, so $i_*$ is also a surjective map from the space of $\Gamma$-indecomposable elements of $\I  \Tor^U(k, k)$ onto that of $\I  \Tor^V(k, k)$. Taking $k$-vector space duals, we observe that $i$ induces an injective Lie algebra homomorphism $i^* : \pi(V) \hookrightarrow \pi(U)$.
Since $\pi(U)$ is a free Lie algebra, the Lie algebra $\pi(V)$ is also free, so $V$ is a Golod algebra. This implies that $R$ is a generalized Golod ring of level $n$.
\end{proof}

\begin{theorem}\label{ggc4}
Let $(R, \mathfrak{m}, k)$ be an Artinian Gorenstein local ring, and let $l \geq 2$. Then $R$ is a generalized Golod ring of level $l$ if and only if $R / \soc (R)$ is a generalized Golod ring of level $l$.
\end{theorem}
\begin{proof}
We use the notation established in the proof of Proposition \ref{ggc3}. Any Artinian  local ring of embedding dimension $1$ is Golod, so we  may assume that $\edim(R) = n  \geq 2$. 

First, we assume that $\bar{R} = R / \soc (R)$ is a generalized Golod ring of level $l$. Then $V^l$ is a Golod algebra. 
We consider the commutative diagram below. 

\[
\begin{tikzpicture}[>=stealth]
  \node (A) at (0,0) {$\Ho(\m V^l)$};
  \node (B) at (3,0) {$\Ho(\m \bar{R}\langle X, Y, Z \rangle)$};
  \node (C) at (3,2) {$\Ho(\m R\langle X \rangle)$};
  \node (D) at (0,2) {$\Ho(\m U^l)$};
  \node (E) at (-2,1) {$\Ho(\m \bar{U}^l)$}; 

  \draw[->] (A) -- (B);
  \draw[->, right hook->] (C) -- (B);
  \draw[->, right hook->] (D) -- (A);
  \draw[->, right hook->] (E) -- (A);

 \draw[->] (D) -- (C); 
  \draw[->, left hook->] (D) -- (E);             
\end{tikzpicture}
\]

The bottom horizontal map is induced by the inclusion $\m V^l \hookrightarrow \m R\langle X, Y, Z \rangle$ and is injective  by Theorem \ref{prl7}. 
It follows from property (2) of Proposition \ref{ggc3} that both downward vertical maps are injective. Therefore,  the top horizontal map is injective, consequently  $U^l$ is a Golod algebra. The ring $R$ is a generalized Golod ring of level $l$ by  Theorem \ref{prl7}. 

Conversely, now we assume that $R$ is a generalized Golod ring of level $l$. This implies that the top horizontal map is injective. To prove  that $\bar{R}$ is a generalized Golod ring of level $l$, it suffices to show that the bottom horizontal map is injective. 

Let $\alpha \in \Zi(\m V^l)$ be a boundary in $\m \bar{R}\langle X, Y, Z \rangle$. Then $\alpha = \partial(\beta)$, where $\beta \in \m \bar{R}\langle X, Y, Z \rangle$.  Note that $\beta \in \m V^N$ for large $N$. We show that $\alpha$ is a boundary in $\m V^l$ by considering two cases. 

\noindent
{\bf Case  1}: $\alpha$ is free of $Y_j$, $Z_k$, i.e., $\alpha \in \m \bar{U}^l$. 

We have seen in the proof of Proposition \ref{ggc3} that the map $\Ho(\m \bar{U}^N) \rightarrow \Ho(\m V^N)$ is injective.  
Since the homology class of $\alpha$ is zero in $\Ho(\m V^N)$, it is also zero in $\Ho(\m \bar{U}^N)$.  Therefore, we may assume $\beta \in \m \bar{U}^N$ satisfying  $\alpha = \partial(\beta)$.  Let  $\tilde{\alpha}$ and $\tilde{\beta}$ be lifts of $\alpha$ and $\beta$ in $\m U^l$, $\m U^N$ respectively. Then $\tilde{\alpha} = \partial(\tilde{\beta}) + w$, $w \in \soc(R) R \langle X \rangle$. This shows that $\tilde{\alpha} \in \Zi(\m U^l)$. We see that $\alpha$ is a boundary in $\m \bar{U}^N$, so $\tilde{\alpha}$ is a boundary in $\m U^N$ by (4) of Lemma \ref{sbr1}. The top horizontal map in the diagram is injective, so $\tilde{\alpha}$ is a boundary in $\m U^l$. This shows that $\alpha$ is a boundary in $\m \bar{U}^l$ and  in particular in $\m V^l$.

\noindent
{\bf Case 2}: $\alpha$ is arbitrary.

For convenience we rename the variables $Y_j$, $Z_k$, $\deg(Y_j), \deg(Z_k) \leq N$ to $W_i, 1 \leq i \leq p_N + q_N$, we have $V^l = \bar{U}^l \langle W_i : 1 \leq i \leq p_l + q_l \rangle$. Now $\alpha$ can be expressed as a linear combination of $\Gamma$-monomials in $W_i$ with coefficients in $\bar{U}^l $.
Let $\init(\alpha) = M$ and the leading coefficient of $\alpha$ be $f \in \m \bar{U}^l$.
Note that by Lemma \ref{ggc2}, $\pm f = \nu_{M}(\alpha) =  \nu_{M}(\partial(\beta)) = \pm \partial(\nu_{M}(\beta))$, $\nu_{M}(\beta) \in \m V^N$, so $f$ is a boundary in $\m \bar{U}^l$ by Case  1, i.e., $f = \partial(g)$ for some $g \in \bar{U}^l$. The initial $\Gamma$-monomial of  $\alpha - \partial(gM)$ is less than $M$. 
If we repeat the above argument for $\alpha - \partial(gM)$ in place of $\alpha$, we have after a finite number of  steps a $\beta' \in \m V^l$ such that $\alpha - \partial(\beta') = \alpha' \in \m \bar{U}^l$. Since $\alpha$ is a boundary in $\m V^N$, so is $\alpha'$.  
By Case 1, $\alpha'$ is a boundary in $\m \bar{U}^l$.  This shows that $\alpha$  is a boundary in $\m V^l$. Therefore, the proof follows. 
\end{proof}

\begin{theorem}\label{ggc5}
	Let $(R, \m_R, k)$ and $(S, \m_S, k)$ be two Artinian Gorenstein local rings. Set $T = R \# S$ and assume $l \geq 2$. Then, $T$ is a generalized Golod ring of level $l$ if and only if both $R$ and $S$ are.
\end{theorem}
\begin{proof}
	We have $\frac{T}{\soc(T)} = \frac{R}{\soc(R)} \times_k \frac{S}{\soc(S)}$. Hence the result follows from Theorems  \ref{ggc4}, \ref{ggfl4}. 
\end{proof}

\section{Applications and Further Remarks}\label{sec:ben3}
The structure of a Gorenstein local ring $(R, \m, k)$ is often reflected in the properties of its maximal ideal $\m$. 
In what follows, we present a criterion for the decomposability of Gorenstein local rings as connected sums.

\begin{theorem}\label{ac3}
Let $(R, \m, k)$ be an Artinian Gorenstein local ring with Loewy length $\lo(R) \geq 3$. Then the following are equivalent:

\begin{enumerate}
\item
There exist Artinian Gorenstein local rings $(S, \p)$ and $(T, \q)$ with $\lo(S) = \lo(R)$ and $\lo(T) = 2$ such that $R = S \# T$.

\item
$(0 : \m^2) \not \subseteq \m^2$.
\end{enumerate}
\end{theorem}

\begin{proof}
First, we assume that (1) holds. Since $\lo(S), \lo(T) \geq 2$, we have $\m/ \m^2 = \p/ \p^2 \oplus \q / \q^2$. We choose $y \in \q \setminus \q^2$. Since $\q^3 = 0$, we have $(0, y) \m^2 = 0$. But $(0, y) \in \m \setminus \m^2$. Hence (2) follows.

Conversely, we assume that (2) holds. We consider the vector space $\frac{(0 : \m^2)}{\m^2 \cap (0 : \m^2)}$ over $k$ and choose elements $y_1, \ldots, y_n \in (0 : \m^2)$ such that their images in $\frac{(0 : \m^2)}{\m^2 \cap (0 : \m^2)}$ form a basis. Let $I = (y_1, \ldots, y_n)$. Then it is easily seen that $I + [\m^2 \cap (0 : \m^2)] =  (0 : \m^2)$ and $\m^2  \cap I \subseteq \soc(R)$. 

We shall show that $I\cap (0 : I) = \soc(R)$. Let $z \in I \cap (0 : I)$. Then $z$ annihilates both $I$ and $\m^2$, so it annihilates $I + [\m^2 \cap (0 : \m^2)] = (0 : \m^2)$. This implies that $z \in (0 : (0 : \m^2)) = \m^2$, since $R$ is a Gorenstein ring. It follows that  $z \in \m^2 \cap I$.  Now, because $\m^2 \cap I \subseteq \soc(R)$, we have $z \in \soc(R)$. Therefore, we conclude that  $I \cap (0 : I) \subseteq \soc(R)$. 
If  $I \cap (0 : I) = 0$, then $\ell(I + (0 : I)) = \ell(I) + \ell(0 : I) = \ell(R)$ which is  a contradiction since $I + (0 : I)$ is a proper ideal. Therefore, $I \cap (0 : I) = \soc(R)$. 

Now 
\[\ell(I + (0 : I)) = \ell(I) + \ell(0 : I) - \ell(I \cap (0 : I)) = \ell(R) - 1 = \ell(\m).\]
Moreover, $I + (0 : I) \subseteq \m$. Therefore, $I + (0 : I) = \m$ and we obtain the decomposition 
\begin{equation}\label{deq}
\frac{\m}{ \soc(R)} = \frac{I}{ \soc(R)} \oplus \frac{(0 : I)}{\soc(R)}.
\end{equation}

Observe that  $\m^2 I = 0$, so $\m I \subseteq \soc(R)$. If $\m I = 0$, then $I \subseteq \soc(R) \subseteq \m^2$, which implies that $(0 : \m^2) = I + [\m^2 \cap (0 : \m^2)] \subseteq \m^2$, a contradiction to our hypothesis.  Therefore, $\m I = \soc(R)$. 

If $\m (0: I) = 0$, then $\m^2 = \m[I + (0 : I)] = \m I = \soc(R)$ which is again a contradiction since $\lo(R) \geq 3$. Therefore, $\soc(R) \subseteq \m (0 : I)$. 

Tensoring equation \eqref{deq} with $R/ \m$, we obtain $\frac{\m}{ \m^2} = \frac{I}{ \m I} \oplus \frac{(0 : I)}{\m (0 : I)}$. 
The ideal $I$ is minimally generated by $y_1, \ldots, y_n$.
We choose a minimal generating set $\{x_1, \ldots, x_m\}$ of $(0 : I)$ such that $\{x_1, \ldots, x_m, y_1, \ldots, y_n\}$ is a minimal generating set of $\m$.

Note that $(x_1, \ldots, x_m)(y_1, \ldots, y_n) = 0$. We have $I^3 = 0$ since $\m^2 I = 0$. If $I^2 = 0$, then $ \m I =  [I + (0 : I)] I= 0$, a contradiction. Therefore,  $\max\{ i : I^i \neq 0\} = 2$. We have $\m I = \soc(R) \subseteq \m (0 : I)$, so 
\[\m^2 = \m [I + (0 : I)]  = \m (0 : I) = [I + (0 : I)] (0 : I) =  (0 : I)^2.\] 
By a straightforward inductive argument, $\m^n = (0 : I)^n$ for all $n \geq 2$. 
Since $\{x_1, \ldots, x_m\}$ is a minimal generating set of $(0 : I)$, we have 
\[\max\{ i : (x_1, \ldots, x_m)^i \neq 0\} = \max\{ i : \m^i \neq 0\} = \lo(R).\]
Therefore, the result follows from \cite[Proposition 4.1]{ananthnarayan2019decomposing} and fact (2) of Definition \ref{prl10}.
\end{proof}

Applying the connected sum decomposition criterion from Theorem~\ref{ac3}, we obtain sufficient conditions under which a Gorenstein local ring is generalized Golod of level 2.
\begin{theorem}\label{ggc6}
Let $(R, \m)$ be a Gorenstein local ring. Then $R$ is a generalized Golod ring of level $2$, and therefore good in the following cases.
\begin{enumerate}
\item
$\m^4 = 0$ and $\mu(\m^2) \leq 4$.

\item
The multiplicity of $R$ is at most $12$ and its $h$-vector is different from $(1, 5, 5, 1)$.
\end{enumerate}
\end{theorem}

\begin{proof}
We begin the proof with some general observations. Let $x \in \m \setminus \m^2$ be a nonzero divisor of $R$. Then $\pi(R) \cong \pi(R/(x))$, which implies that $\pi^{> 2}(R) \cong \pi^{> 2}(R/(x))$. It follows that $R$ is a generalized Golod ring of level 2 if and only if so is $R/(x)$. Thus, to prove the result in case $2$, we may assume by a standard reduction argument that $R$ is an Artinian Gorenstein local ring.

Suppose $R$ is an Artinian Gorenstein local ring. If $\lo(R) = 1$ or $\mu(\m) = 1$, then $R$ is a quotient of a discrete valuation ring. Consequently, $R$ is a Golod ring and the proof follows.
If $\lo(R) = 2$ and $\mu(\m) \geq 2$, then $R$ is a stretched Gorenstein ring, i.e., $\mu(\m^2) = 1$. In this case, $R$ is a generalized Golod ring of level $2$ (see \cite[Theorem 5.4, Remark 5.5]{croll2020detecting} or \cite[Lemma 3.8, Theorem I]{gupta2020criterion}). If $\mu(\m) \leq 4$, then also $R$ is a generalized Golod ring of level $2$ \cite[Theorem 6.4]{avramov1988poincare}.
Therefore, to complete the proof in both cases, it suffices to assume further that $\lo(R) \geq 3$ and $\mu(\m) \geq 5$.

We have $\ell(0 : \m^2) = \ell(R/ \m^2) = 1 + \ell(\m/ \m^2) \geq 6$. In case (1), we have
\[\ell(\m^2) = \ell(\m^2/ \m^3) + \ell(\m^3) \leq 5,\]
therefore, $(0 : \m^2) \not\subseteq \m^2$.
	
Next we consider case (2). By the hypothesis $\ell(R) \leq 12$. We claim $(0 : \m^2) \not\subseteq \m^2$. Suppose on the contrary, we have $(0 : \m^2) \subseteq \m^2$. Then
\[\ell(R/ \m^2) = \ell(0 : \m^2) \leq \ell(\m^2) = \ell(R) - \ell(R/ \m^2).\]
It follows that $\ell(R/ \m^2) \leq \frac{1}{2} \ell(R) \leq 6$, but by our assumption 
\[\ell(R/ \m^2) = 1 + \ell(\m/ \m^2) = 1 + \mu(\m) \geq 6.\] Therefore, $\ell(0 : \m^2) = \ell(R/ \m^2) =6$. Now $\ell(\m^2) = \ell(R) - \ell(R/ \m^2) \leq 6 = \ell(0 : \m^2)$ and $(0 : \m^2) \subseteq \m^2$. Therefore, we must have $(0 : \m^2) = \m^2$, consequently $\m^4 = 0$. Since $R$ is a Gorenstein ring,  it follows that  $\mu(\m^3) = \ell(\m^3) = 1$. We also have
\begin{align*}
\ell(\m/ \m^2) &= \ell(R/\m^2) - 1 = 5 \ \text{and} \\
\ell(\m^2/ \m^3) &= \ell(\m^2) - \ell(\m^3) = \ell(0 : \m^2) - \ell(\m^3) = 6 - 1 = 5
\end{align*}
Therefore, the $h$-vector is $(1, 5, 5, 1)$, which is a contradiction. Hence in this case too we have $(0 : \m^2) \not\subseteq \m^2$.

It follows from Theorem \ref{ac3} that in both cases $R = S \# T$ for some Gorenstein local rings $(S, \p)$, $(T, \q)$ such that $\lo(S) = \lo(R) \geq 3$ and $\lo(T) = 2$. We have $\ell(R) = \ell(S) + \ell(T) - 2$. Since $\lo(T) = 2$, we have $\ell(T) \geq 3$ and consequently $\ell(S) < \ell(R)$.

If $\mu(\p) \geq 5$, then $S$ can be further decomposed as a connected sum by repeating the same argument as applied to $R$. Thus, in both cases $R$ can be expressed as the connected sum of finitely many Gorenstein local rings of Loewy length $2$ and an Artinian Gorenstein local ring of embedding dimension at most $4$. Hence the result follows by applying \cite[Theorem 6.4]{avramov1988poincare}, \cite[Theorem 5.4, Remark 5.5]{croll2020detecting} or   \cite[Lemma 3.8, Theorem I]{gupta2020criterion} and Theorem \ref{ggc5}.	
\end{proof}

\begin{example}\label{noce}
We construct an Artinian Gorenstein $\QQ$-algebra $R$ of multiplicity $12$ with $h$-vector $(1,5,5,1)$ that is indecomposable as a connected sum.

Let $F = x_3^2x_4 + x_2x_4^2 + x_1^2x_5 + x_2^2x_5 	\in \QQ[x_1,x_2,x_3,x_4,x_5]$ be a homogeneous polynomial of degree $3$, and set
$R = \QQ[x_1,x_2,x_3,x_4,x_5]/\ann(F),$
where $\ann(F)$ denotes the apolar ideal of $F$ defined via Macaulay’s inverse system. A direct computation in \textsc{Macaulay2} \cite{InverseSystemsSource} shows that
\[ \ann(F) = (x_1x_2,\; x_1x_3,\; x_1x_4,\; x_2^2 - x_1^2,\; x_2x_3,\; x_2x_4 - x_3^2,\; x_2x_5 - x_4^2,\; x_3x_5,\; x_4x_5,\; x_5^2). \]
From this description, one verifies that $R$ is an Artinian Gorenstein ring with Hilbert function $(1,5,5,1)$ and therefore of multiplicity $12$.

Since $F$ is homogeneous of degree $3$ and $\ann(F)$ has no minimal generator of degree $3$, it follows from \cite[Theorem~1.1]{buczynska2015apolarity} that, after any linear change of coordinates, $F$ cannot be written as a sum of two nonzero polynomials in disjoint sets of variables. Consequently, $R$ cannot be expressed as a nontrivial connected sum of Artinian Gorenstein $\QQ$-algebras.

This example shows that the method used in Theorem~\ref{ggc6} does not extend to Artinian Gorenstein local rings of multiplicity $12$ with $h$-vector $(1,5,5,1)$.
\end{example}

In our final theorem, we provide examples of Artinian Gorenstein rings that are not generalized Golod for every multiplicity greater than $18$.

\begin{theorem}\label{nonGGRAG}
	There exists an Artinian Gorenstein ring $S_n$ with $h$-vector $(1,n,n,1)$ which is not generalized Golod for each $n \geq 8$. 
	In particular, for each $m \geq 18$, there exists an Artinian Gorenstein ring of multiplicity $m$ that is not generalized Golod.
\end{theorem}

\begin{proof}
Roos proved that the ring 
\[
R = k[x,y,z,u]/(x^2, y^2, z^2, u^2, xy, uz)
\]
is a Koszul algebra with $h$‑vector $(1,4,4)$. However, it is not a good Koszul algebra and, in particular, not a generalized Golod ring \cite[Theorem~2.4(A)]{MR2158761}. Using the idealization technique \cite[Lemma, p.~231]{gulliksen1970massey}, we construct an Artinian Gorenstein ring $\widetilde{R}: = R \ltimes \E_R(k)$, where $\E_R(k)$ denotes the injective hull of $k$ over $R$. The $h$‑vector of $\widetilde{R}$ is $(1,8,8,1)$.

The natural surjection $\widetilde{R} \to R$ is a large homomorphism \cite[Theorem~1.1]{levin1980large}, \cite[Theorem~1]{herzog1977algebra}. By \cite[Proposition~2.5]{MR2158761}, the homomorphic image of a good ring under a large homomorphism is good. Since $R$ is not good, $\widetilde{R}$ is not good. In particular, $\widetilde{R}$ is not a generalized Golod ring.

The connected sum $S_9:= \widetilde{R} \# \frac{k[t]}{(t^4)}$ is an Artinian Gorenstein ring with $h$-vector \( (1,9,9,1) \). Since $ \widetilde{R}$ is not a generalized Golod ring, the connected sum \( S_9 \) is also not a  generalized Golod ring by Theorem~\ref{ggc5}. Thus, $S_9$ is an example of an Artinian Gorenstein ring with $h$‑vector $(1,9,9,1)$ that is not generalized Golod.
	
	Continuing this process iteratively by taking connected sums with \( k[t]/(t^4) \), we obtain a family of Artinian Gorenstein rings \( S_n \) with $h$-vector \( (1,n,n,1) \) for $n \geq 9$, each of which is not a generalized Golod ring.
\end{proof}

\begin{acknowledgement}
This paper is part of the PhD thesis of the second author. He acknowledges financial support: file number 09/1020(0176)/2019-EMR-I from Council of Scientific and Industrial Research for his PhD. 
 \end{acknowledgement}


\begin{thebibliography}{10}

\bibitem{ananthnarayan2012connected}
{\sc Ananthnarayan, H., Avramov, L.~L., and Moore, W.~F.}
\newblock Connected sums of gorenstein local rings.
\newblock {\em Journal f{\"u}r die reine und angewandte Mathematik (Crelles
  Journal) 2012}, 667 (2012), 149--176.

\bibitem{ananthnarayan2019decomposing}
{\sc Ananthnarayan, H., Celikbas, E., Laxmi, J., and Yang, Z.}
\newblock Decomposing gorenstein rings as connected sums.
\newblock {\em Journal of Algebra 527\/} (2019), 241--263.

\bibitem{anick1982counterexample}
{\sc Anick, D.~J.}
\newblock A counterexample to a conjecture of serre.
\newblock {\em Annals of Mathematics 115}, 1 (1982), 1--33.

\bibitem{avramov1983local}
{\sc Avramov, L.~L.}
\newblock {\em Local algebra and rational homotopy}.
\newblock Department of Mathematics, University of Stockholm, Sweden, 1983.

\bibitem{MR1165314}
{\sc Avramov, L.~L.}
\newblock Problems on infinite free resolutions.
\newblock In {\em Free resolutions in commutative algebra and algebraic
  geometry ({S}undance, {UT}, 1990)}, vol.~2 of {\em Res. Notes Math.} Jones
  and Bartlett, Boston, MA, 1992, pp.~3--23.

\bibitem{avramov1994local}
{\sc Avramov, L.~L.}
\newblock Local rings over which all modules have rational poincar{\'e} series.
\newblock {\em Journal of Pure and Applied Algebra 91}, 1-3 (1994), 29--48.

\bibitem{avramov1998infinite}
{\sc Avramov, L.~L.}
\newblock Infinite free resolutions.
\newblock In {\em Six lectures on commutative algebra}. Springer, 1998,
  pp.~1--118.

\bibitem{avramov2006golod}
{\sc Avramov, L.~L.}
\newblock Golod homomorphisms.
\newblock In {\em Algebra, Algebraic Topology and their Interactions:
  Proceedings of a Conference held in Stockholm, Aug. 3--13, 1983, and later
  developments\/} (2006), Springer, pp.~59--78.

\bibitem{avramov1971homology}
{\sc Avramov, L.~L., and Golod, E.~S.}
\newblock The homology of algebra of the koszul complex of a local gorenstein
  ring.
\newblock {\em Mat. Zametki 9}, 1 (1971), 53--58.

\bibitem{avramov1988poincare}
{\sc Avramov, L.~L., Kustin, A.~R., and Miller, M.}
\newblock Poincar{\'e} series of modules over local rings of small embedding
  codepth or small linking number.
\newblock {\em Journal of Algebra 118}, 1 (1988), 162--204.

\bibitem{bogvad1983gorenstein}
{\sc B{\o}gvad, R.}
\newblock Gorenstein rings with transcendental poincar{\'e}-series.
\newblock {\em Mathematica Scandinavica 53}, 1 (1983), 5--15.

\bibitem{bruns1998cohen}
{\sc Bruns, W., and Herzog, H.~J.}
\newblock {\em Cohen-macaulay rings}.
\newblock No.~39. Cambridge university press, 1998.

\bibitem{buczynska2015apolarity}
{\sc Buczy{\'n}ska, W., Buczy{\'n}ski, J., Kleppe, J., and Teitler, Z.}
\newblock Apolarity and direct sum decomposability of polynomials.
\newblock {\em Michigan Mathematical Journal 64}, 4 (2015), 675--719.

\bibitem{croll2020detecting}
{\sc Croll, A., Dellaca, R., Gupta, A., Hoffmeier, J., Mukundan, V.,
  {\c{S}}ega, L.~M., Sosa, G., Thompson, P., and Tracy, D.~R.}
\newblock Detecting koszulness and related homological properties from the
  algebra structure of koszul homology.
\newblock {\em Nagoya Mathematical Journal 238\/} (2020), 47--85.

\bibitem{InverseSystemsSource}
{\sc Eisenbud, D., and Boij, M.}
\newblock {InverseSystems: equivariant Macaulay inverse systems. Version~1.1}.
\newblock A \emph{Macaulay2} package available at
  \url{https://github.com/Macaulay2/M2/tree/stable/M2/Macaulay2/packages}.

\bibitem{ghione1975some}
{\sc Ghione, F., and Gulliksen, T.~H.}
\newblock Some reduction formulas for the poincar{\'e} series of modules.
\newblock {\em Atti della Accademia Nazionale dei Lincei. Classe di Scienze
  Fisiche, Matematiche e Naturali. Rendiconti 59\/} (1975), 82--91.

\bibitem{gulliksen1970massey}
{\sc Gulliksen, T.~H.}
\newblock Massey operations and the poincar{\'e} series of certain local rings.
\newblock {\em Preprint series: Pure mathematics http://urn. nb. no/URN: NBN:
  no-8076\/} (1970).

\bibitem{gulliksen1974change}
{\sc Gulliksen, T.~H.}
\newblock A change of ring theorem with applications to poincar{\'e} series and
  intersection multiplicity.
\newblock {\em Mathematica Scandinavica 34}, 2 (1974), 167--183.

\bibitem{gulliksen1969homology}
{\sc Gulliksen, T.~H., and Levin, G.}
\newblock Homology of local rings.
\newblock {\em (No Title)\/} (1969).

\bibitem{gupta2020criterion}
{\sc Gupta, A.}
\newblock A criterion for modules over gorenstein local rings to have rational
  poincar{\'e} series.
\newblock {\em Pacific Journal of Mathematics 305}, 1 (2020), 165--187.

\bibitem{herzog1977algebra}
{\sc Herzog, J.}
\newblock Algebra retracts and poincar{\'e}-series.
\newblock {\em manuscripta mathematica 21\/} (1977), 307--314.

\bibitem{lemaire2006algebres}
{\sc Lemaire, J.-M.}
\newblock {\em Algebres connexes et homologie des espaces de lacets}, vol.~422.
\newblock Springer, 2006.

\bibitem{lescot1983serie}
{\sc Lescot, J.}
\newblock La s{\'e}rie de bass d’un produit fibr{\'e} d’anneaux locaux.
\newblock In {\em S{\'e}minaire d’Alg{\`e}bre Paul Dubreil et Marie-Paule
  Malliavin: Proceedings, Paris 1982 (35{\`e}me Ann{\'e}e)\/} (1983), Springer,
  pp.~218--239.

\bibitem{levin1976lectures}
{\sc Levin, G.}
\newblock Lectures on golod homomorphisms.
\newblock {\em (No Title)\/} (1976).

\bibitem{levin1980large}
{\sc Levin, G.}
\newblock Large homomorphisms of local rings.
\newblock {\em Mathematica Scandinavica 46}, 2 (1980), 209--215.

\bibitem{levin1978factoring}
{\sc Levin, G.~L., and Avramov, L.~L.}
\newblock Factoring out the socle of a gorenstein ring.
\newblock {\em Journal of Algebra 55}, 1 (1978), 74--83.

\bibitem{moore2009cohomology}
{\sc Moore, W.~F.}
\newblock Cohomology over fiber products of local rings.
\newblock {\em Journal of Algebra 321}, 3 (2009), 758--773.

\bibitem{roos1979relations}
{\sc Roos, J.-E.}
\newblock Relations between the poincar{\'e}-betti series of loop spaces and of
  local rings.
\newblock In {\em S{\'e}minaire d'Alg{\`e}bre Paul Dubreil: Proceedings, Paris
  1977--78 (31{\`e}me Ann{\'e}e)\/} (1979), Springer, pp.~285--322.

\bibitem{MR2158761}
{\sc Roos, J.-E.}
\newblock Good and bad {K}oszul algebras and their {H}ochschild homology.
\newblock {\em J. Pure Appl. Algebra 201}, 1-3 (2005), 295--327.

\bibitem{MR201468}
{\sc Serre, J.-P.}
\newblock {\em Alg\`ebre locale. {M}ultiplicit\'es}, vol.~11 of {\em Lecture
  Notes in Mathematics}.
\newblock Springer-Verlag, Berlin-New York, 1965.
\newblock Cours au Coll\`ege de France, 1957--1958, r\'edig\'e{} par Pierre
  Gabriel, Seconde \'edition, 1965.

\bibitem{tate1957homology}
{\sc Tate, J.}
\newblock Homology of noetherian rings and local rings.
\newblock {\em Illinois Journal of Mathematics 1}, 1 (1957), 14--27.

\end{thebibliography}
\end{document}